\documentclass{amsart}
\usepackage{amsfonts}
\usepackage{a4wide}
\usepackage{amssymb}
\usepackage{latexsym}
\usepackage{color}
\usepackage{euscript}  
\numberwithin{equation}{section} \theoremstyle{plain}
\usepackage{units}
\usepackage{euscript}
\usepackage{mathrsfs}
\usepackage{setspace}
\usepackage{bbold}
\makeatletter
\@namedef{subjclassname@2010}{%
\textup{2010} Mathematics Subject Classification}
\makeatother

\newtheorem{thm}{Theorem}[section]
\newtheorem{prop}[thm]{Proposition}
\newtheorem{cor}[thm]{Corollary}
\newtheorem{lem}[thm]{Lemma}

\newtheorem*{hpt*}{Hipoteza}
\newtheorem*{prob*}{Problem}
\newtheorem*{thm*}{Theorem}
\newtheorem*{pro*}{Proposition}
\newtheorem*{met*}{Method}
\newtheorem*{lem*}{Lemma}
\newtheorem{dfn}[thm]{Definition}

\DeclareMathOperator{\D}{d\!} 
 \DeclareMathOperator{\dom}{D}
 \DeclareMathOperator{\M}{m}

\theoremstyle{definition}
\newtheorem{exa}[thm]{{\it Example}}
\newtheorem*{exa*}{{\it Example}}

\theoremstyle{remark}
\newtheorem{rem}[thm]{{\it Remark}}
\newtheorem*{rem*}{{\it Remark}}

\def\rank{\textrm{rank }}
\def\C{\mathbb{C}}
\def\R{\mathbb{R}}

\def\Z{\mathbb{Z}}

\def\A{{\bf{A}}}
\def\B{{\bf{B}}}
\def\w{{\bf{w}}}
\def\u{{\bf{u}}}
\def\p{{\bf{p}}}
\def\a{{\bf{a}}}
\def\s{{\bf{s}}}
\def\b{{\bf{b}}}

\def\kk{\mathcal K}
\def\ff{\mathcal F}
\def\aa{\mathcal{A}}

\def\bb{\mathcal{B}}

\def\bb{\mathcal{B}}

\def\hh{\mathcal H}

\def\xx{\mathcal X}

\def\N{\mathbb N}

\def\munfty{\mu_{\mathtt{G}}}
\newcommand*{\MUN}[1]{\mu_{\mathtt{G},#1}}

\newcommand{\Le}{\leqslant}
\newcommand{\Ge}{\geqslant}
\newcommand*{\bor}[1]{\mathfrak B(#1)}
\newcommand*{\borc}[1]{\mathfrak B_c(#1)}

\newcommand*{\lin}{\mathtt{lin}\,}

\newcommand\frs{\mathfrak{F}}

\newcommand\ca{C_{\A}}
\newcommand*{\is}[2]{\langle#1,#2\rangle}
\newcommand*{\esf}{\mathsf{E}}
\newcommand*{\hsf}{\mathsf{h}}
\begin{document}
\title[Composition operators via inductive limits]{Unbounded composition operators via inductive limits: cosubnormal operators with matrix symbols. II}
\author[P. Budzy\'{n}ski]{Piotr Budzy\'{n}ski}
\author[P.\ Dymek]{Piotr Dymek}
\author[A. P{\l}aneta]{Artur P{\l}aneta}
\address{Katedra Zastosowa\'n Matematyki, Uniwersytet Rolniczy w Krakowie, ul. Balicka 253c, 30-198 Krak\'ow, Poland}
\email{piotr.budzynski@ur.krakow.pl}
\email{piotr.dymek@ur.krakow.pl}
\email{artur.planeta@ur.krakow.pl}

\begin{abstract}
The paper deals with unbounded composition operators with infinite matrix symbols acting in $L^2$-spaces with respect to the gaussian measure on $\R^\infty$. We introduce weak cohyponormality classes $\EuScript{S}_{n,r}^*$ of unbounded operators and provide criteria for the aforementioned composition operators to belong to $\EuScript{S}_{n,r}^*$. Our approach is based on inductive limits of operators.
\end{abstract}
\setstretch{1.0}
\maketitle
\section{Introduction}
Bounded composition operators in $L^2$-spaces are classical object of investigation in operator theory (see the monograph \cite{sin-man}). Their unbounded counterparts have attracted attention quite recently but already proved to be source of interesting problems and results (see \cite{b-d-p-matrix-subn, b-j-j-s-ampa,b-j-j-s-jmaa-14,b-j-j-s-aim, bud-pla-s2, jab}. Many of them are related to the subnormality, a subject widely recognized as difficult and important in operator theory (see the monograph \cite{con} concerning bounded subnormal operators, and trilogy \cite{sto-sza-I, sto-sza-II, sto-sza-III} concerning unbounded ones).

There is no effective general criterion for subnormality of unbounded operators. As a consequence, the methods of verifying the subnormality of an operator depends on its properties. In general, the moment problem approach has been very successful (cf. \cite{sto-sza-moment, c-s-sz}), especially for operators with a dense set of $C^\infty$-vectors. On the other hand, for unbounded composition operators the consistency condition approach (related to a problem of selecting appropriate probability measures) is much better (cf. \cite{b-j-j-s-aim}). This calls for testing various methods when studying the subject. As shown recently, inductive limit techniques might also be helpful in this matter, e.g., in a case of weighted shifts on directed trees or composition operators (cf. \cite{b-j-j-s-jmaa-12, b-j-j-s-jmaa-13,b-d-p-matrix-subn}).

In a recent paper \cite{b-d-p-matrix-subn} we have provided a criterion for cosubnormality of unbounded composition operators induced by finite matrix symbols and also a new proof of the criterion for subnormality of these operators given in \cite{b-j-j-s-aim}. Inductive limit method played a pivotal role in the proofs. A natural setting for generalization and new area of testing our methods is where finite matrix symbols are exchanged by infinite ones. Unbounded composition operators with such symbols have already been investigated in \cite{bud-pla-s2}, where we have dealt with questions of their boundedness and dense definiteness. Motivated by these previous results and the criterion for subnormality of general unbounded operators due to Stochel and Szafraniec (see \cite[Theorem 3]{sto-sza-II}), we introduce in this paper classes $\EuScript{S}_{n,r}^*$ of unbounded operators closely related to cosubnormal operators (they resemble, in a sense, weak hyponormality classes studied in the case of bounded operators, cf. \cite{mcc-pau-2, jun-par-sto,cur-2, jun-lee-par}) and investigate under what conditions composition operators with infinite matrix symbols belong to the classes. We use inductive limits to achieve our goal. This results in three criteria (see Theorem \ref{cosubgauss-} and Propositions \ref{zale-szefa} and \ref{cosubgauss-positive}). The symbol of a composition operator in the first of the criteria has unspecified form whence in the second and the third one the symbol is induced by an infinite matrix. Using this type of matrices allowed us to formulate the criterion in a tractable form, which consequently enabled us to construct explicit examples. To the best of our knowledge none of the examples cannot be studied by other means than our criteria.

The paper is organized as follows. We begin by introducing basic notation and defining the classes $\EuScript{S}_{n,r}$  and $\EuScript{S}_{n,r}^*$ in Section 2. In Section 3 we provide necessary information about composition operators in $L^2$-spaces and their relatives -- weighted composition operators and partial composition operators. Then, in Section 4, we give a brief description of composition operators with finite and infinite matrix symbols in $L^2$-spaces with respect to the gaussian measure. The last part of the paper, Section 5, is devoted to the criteria and examples.
\section{Preliminaries}
Throughout the paper $\Z_+$, $\N$,  $\R$ and $\C$ stand for the set of nonnegative integers, positive integers, real numbers and complex numbers, respectively. For $\kappa\in\N$, $I_\kappa$ stands for the set $\{1,2,\ldots, \kappa\}$, the Cartesian product of $\kappa$-copies of $\R$ is denoted by $\R^\kappa$, and $\R^\infty$ denotes the Cartesian product of $\aleph_0$-copies of $\R$. If $t,s\in\N$ satisfy $s\Ge t$, then by $\pi^s_t$ and $\pi_t$ we denote the mappings $\pi^s_t\colon \R^s\ni(x_1,\ldots, x_s)\mapsto (x_1,\ldots, x_t)\in\R^t$ and $\pi_t\colon \R^\infty\ni(x_1,x_2,\ldots)\mapsto (x_1,x_2,\ldots, x_t)\in\R^t$. If $\{X_n\}_{n=1}^\infty$ is a sequence of subsets of a set $X$ such that $X_k\subseteq X_{k+1}$ for every $k\in N$ and $X=\bigcup_{n=1}^\infty X_n$, then we write $X_n\nearrow X$ as $n\to\infty$. The symmetric difference of sets $A$ and $B$ is denoted by $A\triangle B$. For a topological space $X$, $\bor{X}$ stands for the $\sigma$-algebra of all Borel subsets of $X$. If $\kappa\in \N$ and $\p=\{p_n\}_{n=1}^\infty\subseteq (0,+\infty)$, then $\ell^\kappa(\p)$, or $\ell^\kappa\big(\{p_n\}_{n=1}^\infty\big)$, stands for the weighted $\ell^\kappa$-space $\big\{\{x_n\}_{n=1}^\infty\in\R^\infty\colon \sum_{n=1}^\infty|x_n|^\kappa p_n<\infty\big\}$; $\ell^\kappa(\N)$ denotes the space $\ell^\kappa\big(\{1\}_{n=1}^\infty \big)$.

Let $\hh$ be a (complex) Hilbert space and $T$ be an operator in $\hh$ (all operators are linear in this paper). By $\dom(T)$, $\overline{T}$, and $T^*$ we denote the domain, the closure, and the adjoint of $T$, respectively (if they exists). If $T$ is closable and $\ff$ is a subspace of $\dom(T)$ such that $\overline{T|_\ff}=\overline{T}$, then $\ff$ is said to be a core of $T$. If $T$ is densely defined, $\dom(T) \subseteq \dom(T^*)$ and $\|T^*f\| \Le \|Tf\|$ for all $f \in \dom(T)$, then $T$ is called hyponormal. $T$ is said to be subnormal if $\dom(T)$ is dense in $\hh$ and there exist a complex Hilbert space $\kk$ and a normal operator $N$ in $\kk$ (i.e., $N$ is closed, densely defined and satisfies $N^*N=NN^*$) such that $\hh$ is isometrically embedded in $\kk$ and $Sh = Nh$ for all $h \in \dom(S)$.

Let $n\in\Z_+$, $m\in\N$ and $a=\{a_{p,q}^{i,j}\}_{p,q = 0, \ldots, n}^{i,j=1, \ldots, m} \subset \C$. Then $n_a$ denotes the greatest $n_a\in I_n\cup\{0\}$ satisfying the following condition
\begin{align*}
\text{there exist $i,j\in I_m$ such that $|a_{n_a,q}^{i,j}|+|a_{p,n_a}^{i,j}|>0$ for some $p,q\in I_n$.}
\end{align*}
Clearly, $a_{p,q}^{i,j}=0$ for all $i,j\in I_m$ and $p,q\in I_n$ such that $p>n_a$ and $q>n_a$.
\begin{dfn}
Let $n, r\in\Z_+$. We say that a densely defined operator $T$ in a Hilbert space $\hh$ belongs to the class $\EuScript{S}_{n,r}$ if and only if for every $m\in\N$ and every  $a=\{a_{p,q}^{i,j}\}_{p,q = 0, \ldots, n}^{i,j=1, \ldots, m} \subset \C$,
\begin{align} \label{n1+}
\sum_{i,j=1}^m \sum_{p,q=0}^{n_a} a_{p,q}^{i,j} \lambda^p \bar \lambda^q z_i \bar z_j \Ge 0, \quad \lambda, z_1, \ldots, z_m \in \C,
\end{align}
implies
\begin{align}\label{n3}
\sum_{i,j=1}^m \sum_{p,q=0}^{n_a} \sum_{k,l=0}^r a_{p,q}^{i,j} \is{T^{p+k} f_i^l} {T^{q+l} f_j^k} \Ge 0,
\end{align}
for every finite sequence $\{f_i^k\colon i=1,\ldots, m,\ k=0,\ldots, r\} \subseteq \dom(T^{{n_a+r}})$. In turn, we say that $T$ belongs to $\EuScript{S}_{n,r}^*$ if and only if $T^*$ belongs to $\EuScript{S}_{n,r}$.
\end{dfn}
\begin{rem}
The case of $n=r=0$ is of little interest to us, since every densely defined linear operator in $\hh$ belongs to $\EuScript{S}_{0,0}$ (use the classical Schur lemma). Therefore, in the rest of the paper we tacitly assume that $n+r\Ge 1$.
\end{rem}
The following is essentially contained in \cite[Theorem 3]{sto-sza-II} and \cite[Theorem 29]{c-s-sz} (for the reader convenience we give a sketch of the proof).
\begin{prop}\label{chsub}
Let $S$ be an operator in a complex Hilbert space $\hh$. Then the following conditions are satisfied:
\begin{enumerate}
\item[(i)] if $S$ is subnormal, then $S\in\EuScript{S}_{n,r}$ for all $n,r\in\Z_+$,
\item[(ii)] if $S\in\EuScript{S}_{n,0}$ for every $n\in\Z_+$, then $S|_{\dom^\infty(S)}$ is subnormal.
\end{enumerate}
\end{prop}
\begin{proof}[Sketch of the proof]
(i) Assume that $S$ is a subnormal operator. Fix $n,r\in \Z_+$. For $m\in\N$ and $a=\{a_{p,q}^{i,j}\}_{p,q = 0, \ldots, n}^{i,j=1, \ldots, m} \subset \C$ such that \eqref{n1+} is satisfied, we define the polynomials $p^{i,j}$ of two complex variables $\lambda$ and $\bar\lambda$ by
\begin{align*}
p^{i,j}(\lambda,\bar\lambda)=\sum_{p,q=0}^{n_a} a_{p,q}^{i,j} \lambda^p \bar \lambda^q.
\end{align*}
Let $N$ be a normal extension of $S$ in a Hilbert space $\kk$. Consider $\{f_i^k\}_{i=1,\ldots, m}^{k=0,\ldots, r} \subseteq \dom(S^{{{n_a}+r}})$. Then $\{f_i^k\}_{i=1,\ldots, m}^{k=0,\ldots, r} \subseteq \dom(N^{{{n_a}+r}})$. Since $\dom(N^{*k})=\dom(N^{k*})$ for every $k\in\Z_+$, we have $\dom(N^{{n_a}+r})\subseteq \dom(N^{*r})$ and $N^{*r}\dom(N^{{n_a}+r})\subseteq \dom(N^{n_a})$. This implies that
\begin{equation*}
g_i=\sum_{k=0}^r N^{*k}f_i^k\in\dom(N^{n_a}) \quad \text{for} \quad i\in I_m.
\end{equation*}
Moreover, since $\is{N^{*k}f}{N^{*l} f'}=\is{N^{l}f}{N^{k} f'}$ for all $f,f'\in\dom\big(N^{\max\{k,l\}}\big)$ and $k,l\in\Z_+$, we get
\begin{align*}
\sum_{p,q=0}^{n_a} \sum_{k,l=0}^r a_{p,q}^{i,j} \is{S^{p+k} f_i^l} {S^{q+l} f_j^k}
=\sum_{p,q=0}^{n_a} \sum_{k,l=0}^r a_{p,q}^{i,j} \is{N^{p+k} f_i^l} {N^{q+l} f_j^k}\\
=\sum_{p,q=0}^{n_a} \sum_{k,l=0}^r a_{p,q}^{i,j} \is{N^pN^{*l} f_i^l} {N^qN^{*k} f_j^k}
=\sum_{p,q=0}^{n_a} a_{p,q}^{i,j} \is{N^p g_i} {N^q g_j}.
\end{align*}
Hence, proving (i) amounts to showing that
\begin{equation*}
 \sum_{i,j=1}^m \sum_{p,q=0}^{n_a} a_{p,q}^{i,j}\is{N^p g_i}{N^q g_j}\Ge 0.
 \end{equation*}
Let $E$ be the spectral measure of $N$. For $i,j\in\{1,\ldots m\}$, let $h_{i,j}$ be the Radon-Nikodym derivative of the complex measure $\is{E(\cdot)g_i}{g_j}$ with respect to the non-negative measure $\mu=\sum\limits_{i=1}^m\is{E(\cdot)g_i}{g_i}$. Arguing as in the proof of \cite[Theorem 3]{sto-sza-II} we deduce that 
\begin{align*}
 \sum_{i,j=1}^m \sum_{p,q=0}^{n_a} a_{p,q}^{i,j}\is{N^pg_i}{N^qg_j}=\sum_{i,j=1}^m \int_\C p^{i,j}(\lambda,\bar\lambda)h_{i,j}(\lambda)\, d\mu(\lambda)\Ge 0.
 \end{align*}
This completes the proof of (i).

(ii) This follows directly from \cite[Theorem 29]{c-s-sz}.
\end{proof}
It is worth noticing that every operator in $\EuScript{S}_{n,1}$, $n\in\Z_+$, belongs to the class of hyponormal operators, provided its domain is invariant for the adjoint, in a sense. This follows from the fact that $\EuScript{S}_{n+s,1}\subseteq \EuScript{S}_{n,1}$ for any $s\in\Z_+$ and the following.
\begin{prop}
Let $T$ be densely defined operator such that $T\in\EuScript{S}_{0,1}$. If $\dom(T)\subseteq\dom(T^*)$ and $T^*\dom(T)\subseteq \dom(T)$, then $T$ is hyponormal.
\end{prop}
\begin{proof}
Note that \eqref{n1+} is satisfied with $m=1$, $n=0$ and $a_{0,0}^{1,1}=1$. Then, by \eqref{n3} with $r=1$, we have $\is{f}{f}+\is{g}{Tf}+\is{Tf}{g}+\is{Tg}{Tg}\Ge 0$ for all $f,g\in \dom(T)$. Substituting $f=-T^*g$ into this inequality we get $\|T^*g\|\Le\|Tg\|$ for every $g\in\dom(T)$, which completes the proof.
\end{proof}
\section{Composition operators}
Let $(X,\aa,\mu)$ be a $\sigma$-finite measure space. Let $\A$ be an $\aa$-measurable transformation of $X$. Define the measure $\mu\circ \A^{-1}$ on $\aa$ by setting $\mu\circ \A^{-1} (\sigma)=\mu(\A^{-1}(\sigma))$, $\sigma\in\aa$. If $\A$ is nonsingular, i.e., $\mu\circ \A^{-1}$ is absolutely continuous with respect to $\mu$, then the operator
    \begin{align*}
    \ca\colon L^2(\mu) \supseteq \dom(\ca) \to L^2(\mu)
    \end{align*}
given by
   \begin{align*}
    \dom(\ca)=\{f \in L^2(\mu) \colon f\circ \A \in L^2(\mu)\} \text{ and } \ca f=f\circ \A \text{ for } f\in \dom(C_\A),
   \end{align*}
is well defined and closed in $L^2(\mu)$ (cf. \cite[Proposition 1.5]{b-j-j-s-ampa}), where $L^2(\mu)=L^2(X,\aa,\mu)$ is the space of all $\aa$-measurable $\C$-valued functions with $\int_X|f|^2\D\mu<\infty$. Call it a {\em composition operator} induced by $\A$; we say that $\A$ is the {\em symbol} of $\ca$ then. If the Radon-Nikodym derivative
    \begin{align*}
    \hsf_{\bf \A}=\frac{\D \mu\circ \A^{-1}}{\D\mu}
    \end{align*}
belongs to $L^\infty(\mu)$, the space of all essentially bounded $\aa$-measurable $\C$-valued functions on $X$, then $\ca$ is bounded on $L^2(\mu)$ and $ \|\ca\|=\|\hsf_{\bf A}\|_{L^\infty(\mu)}^{1/2}$. The reverse is also true. By the measure transport theorem we get
    \begin{align*}
    \dom(\ca)=L^2((1+\hsf_{\bf A})\D\mu).
    \end{align*}
It follows from \cite[Proposition 3.2]{b-j-j-s-ampa} that
    \begin{align}\label{dense}
    \text{$\overline{\dom(\ca)}=L^2(\mu)$ if and only if $\hsf_{\bf A}<\infty$ a.e.\ $[\mu]$}.
    \end{align}
As shown below, in case of finite measure spaces dense definiteness of $\ca$ is automatic.
\begin{lem}\label{densepower}
If $(X,\aa,\mu)$ is a finite measure space and $\A$ is a nonsingular transformation of $X$, then $\chi_\sigma\in\dom(\ca^n)$ for every $\sigma\in\aa$ and $n\in\N$. Moreover, $\dom^\infty(\ca)$ is dense in $L^2(\mu)$.
\end{lem}
\begin{proof}
Since $\mu(X)<+\infty$ and $\chi_\sigma\circ\A=\chi_{\A^{-1}(\sigma)}$, we deduce that $\chi_\sigma\in\dom(C_{\A}^n)$ for all $\sigma\in\aa$ and $n\in\N$. Therefore, $\dom(C_{\A}^n)$ is dense in $L^2(\mu)$ for every $n\in\Z_+$. This and \cite[Theorem 4.7]{b-j-j-s-ampa} yield the ``moreover'' part.
\end{proof}

Now we recall some information concerning weighted composition operators. Let $(X,\aa, \nu)$ be a $\sigma$-finite measure space, $\A$ be a nonsingular $\aa$-measurable transformation of $X$ and $\w$ be a $\aa$-measurable $\C$-valued function on $X$ such that the measure $(|\w|^2\D\nu)\circ \A^{-1}$ is absolutely continuous with respect to $\nu$. A {\em weighted composition operator} $W_{\A,\w} \colon L^2(\nu) \supseteq \dom(W_{\A,\w}) \to L^2(\nu)$ is defined by
   \begin{align*}
    \dom(W_{\A,\w}) & = \{f \in L^2(\nu) \colon \w \cdot (f\circ \A) \in L^2(\nu)\},\\
    W_{\A,\w} f & = \w \cdot (f\circ \A), \quad f \in \dom(W_{\A,\w}).
   \end{align*}
Any such operator $W_{\A,\w}$ is closed. Moreover, $W_{\A,\w}$ is densely defined if and only if $\hsf_\A\cdot(\esf_\A (|\w|^2)\circ \A^{-1})<\infty$ a.e. $[\nu]$, where $\esf_\A (\cdot)$ denotes the conditional expectation operator with respect to $\sigma$-algebra $\A^{-1}(\aa)$  (cf. \cite[Lemma 6.1]{cam-hor}). We refer the reader to \cite{b-j-j-s-wco} for more information on unbounded weighted composition operators and references.

The adjoint of a composition operator induced by $\aa$-bimeasurable transformation turns out to be the weighted composition operator induced by the inverse of the symbol (cf. \cite[Corollary 7.3]{b-j-j-s-ampa} and \cite[Remark 7.4]{b-j-j-s-ampa}):
\begin{align}\label{adjoint}
\begin{minipage}{85ex}
{\em If $\A$ is an invertible transformation of $X$ such that both $\A$ and $\A^{-1}$ are $\aa$-measurable and nonsingular, then $\ca^*=W_{\A^{-1}, \hsf_\A}$.}
\end{minipage}
\end{align}
If $\mu$ is a finite measure, then the above implies immediately the following.
\begin{lem}\label{adjointpowers}
Let $n\in\N$. Suppose $(X,\aa,\mu)$ is a finite measure space and $\A$ is an invertible transformation of $X$ such that both $\A$ and $\A^{-1}$ are $\aa$-measurable and nonsingular. Then the following conditions are satisfied:
\begin{enumerate}
\item[(i)] $\displaystyle \dom\big((C_{\A}^*)^n\big)  = \bigcap_{k=1}^n \dom\big(C_{\A^k}^*\big)$,
\item[(ii)] $\big(\ca^*\big)^n\subseteq\big(\ca^n\big)^*=C_{\A^n}^*$,
\item[(iii)] $\displaystyle \big(\ca^*\big)^n=\big(\ca^n\big)^*$, whenever $\displaystyle \dom\big(C_{\A^n}^*\big)=\bigcap_{k=1}^n \dom\big(C_{\A^k}^*\big)$.
\end{enumerate}
\end{lem}
\begin{proof}
Since $ C_\A^n\subseteq C_{\A^n}$ (cf. \cite[Inclusion (3.3)]{b-j-j-s-ampa}), Lemma \ref{densepower} implies that $\ca^n$ and $C_{\A^n}$ are densely defined. This together with \cite[Corollary 4.2.]{b-j-j-s-ampa} yields
\begin{align}\label{adjointpowers+}
\big(\ca^n\big)^*=\big(\overline{\ca^n}\big)^*=C_{\A^n}^*.
\end{align}
By \cite[Lemma 15]{b-j-j-s-aim}, we have
\begin{align}\label{radonAn}
\hsf_{\A^n}=\hsf_\A \cdot \hsf_\A\circ \A^{-1}\cdot \hsf_\A\circ \A^{-2}\cdots \hsf_\A\circ \A^{-(n-1)}\quad \text{ a.e. $[\mu]$},\ n\in\N.
\end{align}

In view of \eqref{adjoint} and \eqref{radonAn}, we obtain the equality
\begin{align}\label{adjointpowers+++}
\dom\big((C_{\A}^*)^n\big)=\bigcap_{k=1}^n\big\{f\in L^2(\mu)\colon \hsf_{\A^k} f\circ \A^{-k}\in L^2(\mu)\big\}=\bigcap_{k=1}^n\dom\big(C_{\A^k}^*\big).
\end{align}
Thus (i) is satisfied.

The fact that $(B^*)^n\subseteq (B^n)^*$ for any operator $B$ in a Hilbert space such that the adjoints exist imply that
\begin{align}\label{adjointpowers++}
\big(\ca^*\big)^n\subseteq \big(C_{\A}^n\big)^*,
\end{align}
which combined with \eqref{adjointpowers+} proves (ii).

The equality in (iii) follows from \eqref{adjointpowers+}, \eqref{adjointpowers+++}, \eqref{adjointpowers++} and the assumption on $\dom\big(C_{\A^n}^*\big)$.
\end{proof}
That the inclusion $(\ca^*)^n\subseteq C_{\A^n}^*$ in Lemma \ref{adjointpowers} can be proper is shown below.
\begin{exa}
Let $X=[0,1]$, $\aa=\bor{[0,1]}$ and $\mu(\sigma)=\int_{\sigma}\frac{1}{\sqrt{x}}\D\M_1$, where $\M_1$ is the Lebesgue measure on $\R$. Let $\A(x)=1-x$. Clearly, $\A$ and $\A^{-1}$ are $\aa$-measurable and nonsingular. Since $\A^2(x)=x$ for every $x\in X$, we get $C_{\A^2}^*=I$, where $I$ denotes the identity operator. On the other hand, by \eqref{adjointpowers+++} we have
\begin{align*}
\dom\big((C_\A^*)^2\big)=\Big\{f\in L^2(\mu) \colon \int_{X} |f(x)|^2 \sqrt{\tfrac{x}{1-x}}\D\mu<\infty\Big\}\neq L^2(\mu).
\end{align*}
This proves that $\dom\big((C_\A^*)^2\big)\neq \dom(C_{\A^2}^*)$.
\end{exa}
Now we show that certain families generated by characteristic functions form cores for $n$-tuples of weighted composition operators (this generalizes \cite[Proposition 3.3]{b-d-p-matrix-subn}). Given operators $T_1, \ldots, T_n$, $n\in\N$, in a Hilbert space $\hh$ and a subspace $\ff\subseteq\bigcap_{i=1}^n\dom(T_i)$, we say that $\ff$ is a core for $(T_1, \ldots, T_n)$ if $\ff$ is dense in $\bigcap_{i=1}^n\dom(T_i)$ with respect to the graph norm $\|f\|_{(T_1, \ldots, T_n)}^2:=\|f\|^2+\sum_{i=1}^n\|T_i f\|^2$.
\begin{prop}\label{rdzen2}
Let $(X,\aa,\mu)$ be a $\sigma$-finite measure space and let $\bb\subseteq\aa$ be a family of sets satisfying the following conditions
\begin{enumerate}
\item[(i)] for all $A, B \in \bb$, $A \cap B \in \bb$,
\item[(ii)] $\aa=\sigma(\bb)$,
\item[(iii)] there exists $\{X_n\}_{n=1}^\infty\subseteq \bb$ such that $X_k\subseteq X_{k+1}$ for every $k\in\N$ and $\mu\big(X\setminus \bigcup_{n=1}^\infty X_n\big)=0$.
\end{enumerate}
Let $n\in\N$. Suppose $\A_i\colon X\to X$, $i \in I_n$, is invertible, $\A_i$ and $\A_i^{-1}$ are $\aa$-measurable and nonsingular. Let $\w_i\colon X\to \C$, $i\in I_n$, be $\aa$-measurable. If $\ff:=\lin\big\{\chi_\sigma\colon \sigma\in\bb\big\}\subseteq \bigcap_{i=1}^n\dom(W_{\A_i,\w_i})$, then $\ff$ is a core of $(W_{\A_1,\w_1}, \ldots, W_{\A_n,\w_n})$.
\end{prop}
\begin{proof}
Let $\hsf_{\A_i,\w_i}$, $i\in I_n$, denote the Radon-Nikodym derivative of the measure $\big(|\w_i|^2\D\mu\big)\circ\A_i^{-1}$ with respect to the measure $\mu$. By \cite[Lemma 3.2]{b-d-p-matrix-subn} and \cite[Proposition 10]{b-j-j-s-wco}, $\hsf_{\A_i,\w_i}<\infty$ a.e. $[\mu]$ for every $i\in I_n$. Therefore, the measure $\big(1+\hsf_{\A_1,\w_1}+\ldots+\hsf_{\A_n,\w_n}\big)\D\mu$ is $\sigma$-finite. Combining \cite[Lemma 3.2]{b-d-p-matrix-subn} and \cite[Proposition 9]{b-j-j-s-wco}, we see that $\ff$ is a core of $(W_{\A_1,\w_1}, \ldots, W_{\A_n,\w_n})$ (see also the proof of \cite[Proposition 3.3]{b-d-p-matrix-subn}).
\end{proof}
It is sometimes convenient to consider composition operators, or even weighted composition operators, induced by partial transformations of $X$, i.e., mappings defined not on the whole of $X$, but on a subset of $X$; such composition operators (resp. weighted composition operators) are occasionally called {\em partial composition operators} (resp. {\em partial weighted composition operators}). Suppose $Y\in\aa$. Let $\B\colon Y\to X$ and $\w\colon Y\to\C$ be $\aa$-measurable having $\aa$-measurable extensions $\hat\B\colon X\to X$ and $\hat \w\colon X\to\C$. If the weighted composition operator $W_{\hat\B,\hat \w}$ is well-defined, then we define the operator $W_{\B,\w}\colon \dom(W_{\B,\w})\to L^2(\mu)$ by
\begin{align*}
W_{\B,\w}=W_{\hat\B,\chi_Y \hat \w}.
\end{align*}
Clearly, if $Y=X$, then the above definition agrees with the previous one given for ``everywhere defined'' transformations, which justifies the notation. The partial composition operator comes out of it, when we consider $\w=\chi_Y$. Let us note that the definition is independent of the choice of extensions (see \cite[Proposition 7]{b-j-j-s-wco}). In particular, we have
\begin{align}
\begin{minipage}{85ex}\label{zgoda}
{\em If $\hat\A\colon X\to X$ is $\aa$-measurable and nonsingular, $\hat \u\colon X\to \C$ is $\aa$-measurable, and $\A$ and $\u$ are their restriction to a full measure $\mu$ subset $Y$ of $X$, then $W_{\hat\A,\hat \u}=W_{\A,\u}$.}
\end{minipage}
\end{align}
In view of \eqref{zgoda}, we see that Lemma \ref{densepower} is still valid if a composition operator $C_\A$ is induced by a (partial) nonsingular measurable transformation $\A\colon Y\to X$ defined on a full measure $\mu$ subset $Y$ of $X$. Moreover, if additionally $\A$ is injective, $\A(Y)$ is a set of full measure $\mu$ and $\A^{-1}\colon \A(Y)\to X$ is nonsingular and measurable, then we get also the claim of Lemma \ref{adjointpowers}. That ``partial'' counterpart of Proposition \ref{rdzen2} is true can be easily proved as well. By \eqref{adjoint} and \eqref{zgoda}, if $\mu$ is finite, then $C_\A$ is densely defined operator in $L^2(\mu)$ and $C_\A^*=W_{\hat\A^{-1}, \hsf_{\hat\A}}$, where $\hat\A^{-1}$ and $\hat\A$ are any $\aa$-measurable extensions of $\A^{-1}$ and $\A$, respectively, onto $X$.

Note that there are transformations, which are invertible and measurable but not nonsingular.
\begin{exa}
Consider $X=\R^\infty$, $\aa=\bor{\R^\infty}$ and $\mu=\munfty$, where $\munfty$ denotes gaussian measure on $\R^\infty$ (see the next Section for details). Let $\A\colon\R^\infty\to \R^\infty$ be given by $\A\{x_n\}_{n=1}^\infty=\big\{\frac1n x_n\big\}_{n=1}^\infty$. It is clear that $\A$ is invertible and measurable. However, $\A$ cannot be nonsingular since $\A^{-1}$ transforms $\ell^2(\N)$ into $\ell^2\big(\{1/n^2\}_{n=1}^\infty\big)$ and, as we know by \cite[Lemma 11 and Theorem 1.3, page 92]{be-ko}, $\munfty\Big(\ell^2(\N)\Big)=0$ while $\munfty\Big(\ell^2\big(\{1/n^2\}_{n=1}^\infty\big)\Big)=1$.
\end{exa}

\section{Composition operators with matrix symbols}
Let $\kappa\in\N$. The {\em $\kappa$-dimensional gaussian measure} is the measure $\MUN{\kappa}$ given by
\begin{align*}
\D\MUN{\kappa}=\frac{1}{(\sqrt{2\pi})^\kappa}\, \exp\bigg({-\frac{x_1^2+\ldots +x_\kappa^2}{2}}\bigg) \D\M_\kappa,
\end{align*}
where $\M_\kappa$ denotes the $\kappa$-dimensional Lebesgue measure on $\R^\kappa$. For any linear transformation $\A$ of $\R^\kappa$, the composition operator $C_\A$ in $L^2(\MUN{\kappa})$ is well-defined if and only if $\A$ is invertible. If this is the case, then (cf.\ \cite[equation (2.1)]{sto})
   \begin{align} \label{pochodna}
    \hsf_\A (x) =|\det \A^{-1}|\cdot \exp\frac{|x|^2-|\A^{-1}x|^2}{2}, \quad x \in \R^\kappa.
   \end{align}
Here, and later on, $|\cdot|$ stands\footnote{For simplicity we use the same dimension-independent symbol for any of the Euclidean norms on $\R^n$, $n\in\N$.} for the Euclidean norm on $\R^n$, $n\in\N$. Combining \eqref{pochodna} with \eqref{dense} and \cite[Proposition 6.2]{b-j-j-s-ampa}, we get:
\begin{align}\label{choinka}
\begin{minipage}{85ex}
{\em Every invertible linear transformation $\A$ of $\R^\kappa$ induces densely defined and injective composition operator $\ca$ in $L^2(\MUN{\kappa})$.}
\end{minipage}
\end{align}
Cosubnormality of $\ca$ can be written in terms of the symbol $\A$ (cf. \cite[Theorem 2.5]{sto} and \cite[Theorem 3.8]{b-d-p-matrix-subn}).
\begin{thm}\label{szarlotka}
Let $\A$ be an invertible linear transformation of $\R^\kappa$. If $\A$ is normal in $(\R^\kappa, |\cdot|)$, then $\ca$ is cosubnormal. The reverse implication holds whenever $\ca$ is bounded on $L^2(\MUN{\kappa})$.
\end{thm}
The {\em gaussian measure $\munfty$ on $\R^\infty$} is the (infinite) tensor product measure
\begin{align*}
\munfty=\MUN{1}\otimes \MUN{1}\otimes\MUN{1}\otimes\ldots
\end{align*}
defined on $\bor{\R^\infty}$. Recall that $\bor{\R^\infty}$ is generated by {\em cylindrical sets}, i.e., sets of the form $\sigma\times\R^\infty$, with $\sigma\in\bor{\R^n}$ and $n\in\N$; the family of all cylindrical sets will be denoted by $\borc{\R^\infty}$. For every $n\in\N$, the space $L^2(\MUN{n})$ can be naturally embedded into $L^2(\munfty)$ via the isometry $\varDelta\colon L^2(\MUN{n})\ni f\mapsto f\circ \pi_n\in L^2(\munfty)$. For simplicity we suppress the explicit dependence on $n$ in the notation. We will denote by symbol $\|\cdot\|$ any of the $L^2$-norm on $L^2(\MUN{n})$, $n\in\N$, or $L^2(\munfty)$. A function $f\in L^2(\munfty)$ is called a {\em cylindrical} function if $f\in \varDelta \Big(L^2(\MUN{n})\Big)$ for some $n\in\N$. If $f\in L^2(\munfty)$ is a cylindrical function and $f=\varDelta \widetilde f$, with $\widetilde f\in L^2(\MUN{k})$ and $k\in\N$, then $\varSigma_l(f)$, $l\Ge k$, denotes the set $\{x\in\R^k\colon \widetilde f(x)\neq0\}\times \R^{l-k}$.

We are interested in properties of composition operators with symbols induced by infinite matrices. Such operators were investigated in \cite[Corollary 5.1]{bud-pla-s2}, where tractable criteria for dense definiteness in case of row-finite matrices were given. In this paper we need to relax our approach slightly and consider also non row-finite case. This can be done as follows. We take a matrix $\a=(a_{ij})_{i,j\in\N}$ with real entries and we set
\begin{align*}
\dom_\a=\Big\{\{x_{n}\}_{n=1}^\infty \in\R^\infty\colon \sum_{j=1}^\infty a_{ij}x_j\in\R \text{ for every }i\in\N\Big\}.
\end{align*}
Using the Cauchy condition we see that
\begin{align*}
\dom_\a=\bigcap_{i\in\N}\bigcap_{l\in\N}\bigcup_{N\in\N}\bigcap_{m\geqslant n \geqslant N} \Big\{\{x_k\}_{k=1}^\infty\colon \big| \sum_{j=n}^m a_{ij}x_j\big|<\tfrac{1}{l}\Big\},
\end{align*}
which yields $\dom_\a\in\bor{\R^\infty}$. Now, if $\A\colon \dom\to \R^\infty$, with $\dom\in\bor{\R^\infty}$ such that $\dom\subseteq \dom_\a$, satisfies
\begin{align*}
\A(x_1,x_2,\ldots)=\bigg(\sum_{j=1}^\infty a_{1j}\,x_j,\sum_{j=1}^\infty a_{2j}\,x_j,\ldots\bigg),\quad \{x_n\}_{n=1}^\infty\in \dom,
\end{align*}
then we say that {\em$\A$ is induced by $\a$}, or that {\em $\a$ induces $\A$}. Such an $\A$ is $\bor{\R^\infty}$-measurable. Indeed, this follows from the fact that sets of the form $\R^{\kappa-1}\times(\alpha,\beta)\times\R^\infty$, with $\kappa\in\N$ and $\alpha,\beta\in\R$, generate $\bor{\R^\infty}$ and
\begin{align*}
\A^{-1}\big( \R^{\kappa-1}\times(\alpha,\beta)\times\R^\infty \big)
= \bigcup_{N\in\N}\bigcap_{m \geqslant N} \Big\{\{x_n\}_{n=1}^\infty\in \dom\colon \alpha<\big| \sum_{j=1}^m a_{\kappa j}x_j\big|<\beta\Big\}.
\end{align*}
The operator $\ca$ is defined according to the scheme of defining composition operators with partial symbols (see the previous Section).

It is obvious that $\dom_\a$ is a set of full measure $\munfty$ for every row-finite matrix $\a$. As shown below, if $\a$ is not row-finite, but entries in each row have appropriate asymptotics, then one can deduce a similar result.
\begin{lem}\label{exponential}
Let $\a=(a_{ij})_{i,j\in\N}$ be a real matrix. If there exist $\lambda\in(0,1)$ and $C>0$ such that $|a_{i,j}|\Le C \lambda^{|i-j|}$ for all $i,j\in\N$, then $\dom_\a$ is a set of full measure $\munfty$.
\end{lem}
\begin{proof}
By the Cauchy-Schwartz inequality $\ell^2\big(\{\lambda^n\}_{n=1}^\infty\big)\subseteq \dom_\a$. Since, by \cite[Lemma 11 and Theorem 1.3, page 92]{be-ko}, $\munfty\Big(\ell^2\big(\{\lambda^n\}_{n=1}^\infty\big)\Big)=1$ we get the claim.
\end{proof}

Our focus will be on composition operators with block 3-diagonal (or, using another language, block 1-banded) matrix symbols. For this we need some extra terminology. Let $\s=\{s(n)\}_{n=1}^\infty$  be an increasing sequence of natural numbers. Let $\a=(a_{ij})_{i,j\in\N}$ be an infinite matrix such that
\begin{align}\label{rc-finite}
a_{ij}=a_{ji}=0,\ (i,j)\in\{s(n-1)+1,\ldots, s(n)\}\times\{s(n+1)+1,\ldots\}, n\in\N,
\end{align}
with $s(0):=0$. Then, for $p,q\in \N$, we define  matrices $\a_{p}=\a_{p}^\s$ and $\a_{pq}=\a_{pq}^\s$ by\allowdisplaybreaks
\begin{align*}
&\a_{p}=(a_{ij})_{i,j=1}^{s(p)}, &&\a_{pp}=(a_{ij})_{i,j=s(p-1)+1}^{s(p)}, \\
&\a_{p+1,p}=(a_{ij})_{i=s(p)+1, j=s(p-1)+1}^{s(p+1), s(p)}, &&\a_{p,p+1}=(a_{ij})_{i=s(p-1)+1, j=s(p)+1}^{s(p), s(p+1)},
\end{align*}
and
\begin{align*}
\a_{pq}=0,\quad |p-q|>1.
\end{align*}
With this notation $\a$ is indeed a block 3-diagonal matrix, i.e.,
\begin{align*}
\a=\left(
  \begin{array}{cccccc}
    \a_{11} & \vline &\a_{12}& \vline & 0 & \cdots \\
    \cline{1-5}
    \a_{21} & \vline &\a_{22}& \vline & \a_{23} & \cdots \\
    \cline{1-5}
    0 & \vline &\a_{32}& \vline & \a_{33} & \cdots \\
    \cdots &  &\cdots & & \cdots & \cdots \\
  \end{array}
\right).
\end{align*}
If additionally to \eqref{rc-finite}, the matrix $\a$ satisfies
\begin{align*}
\rank \a_{p,p+1}\in \big\{0, s(p)-s(p-1)\big\}, \quad p\in\N,
\end{align*}
then we say that $\a$ belongs to the class $\mathfrak{F}(\s)$.
\section{Criteria and examples}
In this section we propose criteria answering the question of when a composition operator $\ca$ in $L^2(\munfty)$ induced by an infinite matrix is of class $\EuScript{S}_{n,r}^*$. The first (see Theorem \ref{cosubgauss-} below) is rather general in nature, while the second and third (see Propositions \ref{zale-szefa} and  \ref{cosubgauss-positive}) are more concrete, enabling us to construct explicit examples.
\begin{thm}\label{cosubgauss-}
Let $n,r\in\Z_+$. Let $\s=\{s(k)\}_{k=1}^\infty\subseteq \N$ be an increasing sequence. Suppose that
\begin{enumerate}
\item[(i)] $\dom\in\bor{\R^\infty}$ satisfies $\munfty(\dom)=1$,
\item[(ii)] $\A\colon \dom\to \A(\dom)\subseteq\R^\infty$ is a $\bor{\R^\infty}$-measurable nonsingular and invertible transformation such that $\munfty\big(\A(\dom)\big)=1$, $\A^{-1}$ is $\bor{\R^\infty}$-measurable and nonsingular,
\item[(iii)] there exist $\kappa\in\N$ and $\{\varOmega_{l}\colon l\in\N\}\subseteq \bor{\R^{s(\kappa)}}$ such that  $\varOmega_{l}\nearrow \R^{s(\kappa)}$ as $l\to\infty$ and $\chi_{\varOmega_{k}\times\R^\infty}\hsf_{\A^i}\in L^2(\munfty)$ for all $k\in \N$ and $i\in I_{n+r}$.
\item[(iv)] for every $k\in\N$ there exists a $\bor{\R^{s(k)}}$-measurable nonsingular and invertible transformation $\A_k$ of $\R^{s(k)}$ such that $\A_k^{-1}$ is $\bor{\R^{s(k)}}$-measurable nonsingular, $C_{\A_k}$ is densely defined and cosubnormal in $L^2(\MUN{s(k)})$,
\item[(v)] for all $i\in I_{n+r}$ and $k\in\N$,
\begin{align}\label{glod}
\limsup_{l\to\infty}\int_{\R^{s(l)}}\chi_{\varOmega_k\times\R^{s(l)-s(\kappa)}}\hsf_{\A_{l}^i}^2\D\MUN{s(l)}\Le\|\chi_{\varOmega_k\times\R^\infty}\hsf_{\A^i}\|^2,
\end{align}
\item[(vi)] for every $m\in\N$ there exists $\widetilde p\in\N$ such that for all $i\in I_{n+r}$ and $\sigma\in\bor{\R^m}$ we have
\begin{align}\label{30stopni-1}
\munfty\bigg(\Big(\A^{-i}(\sigma\times \R^{\infty})\Big)\triangle \Big(\A^{-i}_{p}(\sigma\times \R^{s(p)-m})\times \R^\infty\Big)\bigg)=0, \quad p\Ge \widetilde p.
\end{align}
\item[(vii)] for every $m\in\N$, $\sigma\in\bor{\R^m}$ and $i\in I_{n+r}$ we have
\begin{align}\label{3studentki}
\munfty\bigg(\Big(\A^i(\sigma\times\R^\infty)\Big)\triangle \Big(\bigcup_{k=1}^\infty\bigcap_{p\Ge k} \A_p^i(\sigma\times \R^{s(p)-m})\times \R^\infty \Big)  \bigg)=0.
\end{align}
\end{enumerate}
Then $\ca\in\EuScript{S}_{n,r}^*$.
\end{thm}
\begin{proof}
We divide the proof into few steps.\\
\noindent{\bf Step 1. }{\em For every $m\in\N$ there exists $\widetilde p\in\N$ such that for all $i\in I_{n+r}$ and $\omega=\widetilde\omega\times\R^\infty$, with $\widetilde \omega\in\bor{\R^\eta}$ for some $\eta\in\N$, and every $f=\varDelta \widetilde f\in\dom(C_{\A^i})$, with $\widetilde f\in L^2(\MUN{m})$, we have
\begin{align}\label{adj}
\big|\is{C_{\A^i} f}{\chi_{\omega}}\big|\Le \|f\| \bigg(\int_{\R^{s(l)}} \chi_{\A_{l}^i(\widetilde\omega\times \R^{s(l)-m})\cap\varSigma_{s(l)}(f)} \hsf_{\A_{l}^i}^2\D\MUN{s(l)}\bigg)^{1/2},\quad l\Ge \widetilde p.
\end{align}
}

Take $\sigma=\widetilde\sigma\times \R^{\infty}$ with $\widetilde \sigma\in\bor{\R^k}$ and $k\Ge \eta$. Then, by condition (vi), there exists $\widetilde p\in\N$ such that for all $p\Ge \widetilde p$  and $i\in I_{n+r}$ we have
\allowdisplaybreaks
\begin{align*}
\is{C_{\A^i} \chi_{\sigma}}{\chi_{\omega}}
&=\int_{\R^\infty} (\chi_{\sigma}\circ \A^i) \cdot \chi_{\omega}\D\munfty\\
&=\MUN{s(p)}\Big(\A_p^{-i} (\widetilde\sigma\times \R^{s(p)-k})\cap \widetilde\omega\times \R^{s(p)-\eta}\Big)\\
&=\int_{\R^{s(p)}}\chi_{\widetilde\sigma\times \R^{s(p)-k}}\circ \A_p^i\cdot\chi_{\widetilde\omega\times \R^{s(p)-\eta}}\D\MUN{s(p)}\\
&=\int_{\R^{s(p)}} \chi_{\widetilde\sigma\times \R^{s(p)-k}} \cdot \chi_{\A_p^i(\widetilde\omega\times \R^{s(p)-\eta})} \cdot  \hsf_{\A_p^i}\D\MUN{s(p)}\\
&\Le \|\chi_{\sigma}\| \bigg(\int_{\R^{s(p)}} \chi_{\widetilde\sigma\times \R^{s(p)-k}\cap \A_p^i(\widetilde\omega\times \R^{s(p)-\eta})} \hsf_{\A_p^i}^2\D\MUN{s(p)}\bigg)^{1/2},
\end{align*}
which means that \eqref{adj} holds with $f=\chi_\sigma$. If $f=\varDelta \widetilde f\in\dom(C_{\A^i})$, where $\widetilde f$ is a nonnegative step function, then arguing as above we see that
\begin{align*}
\is{C_{\A^i} f}{\chi_{\omega}}
&=\int_{\R^{s(p)}} \varDelta \widetilde f\,  \chi_{\A_p^i(\widetilde\omega\times \R^{s(p)-\eta})}\,   \hsf_{\A_p^i}\D\MUN{s(p)}\\
&\Le \|f\| \bigg(\int_{\R^{s(p)}} \chi_{\varSigma_{s(p)}(f)\cap \A_p^i(\widetilde\omega\times \R^{s(p)-\eta})} \hsf_{\A_p^i}^2\D\MUN{s(p)}\bigg)^{1/2},\quad p\Ge \widetilde p,\ i\in I_{n+r}.
\end{align*}
In turn, using approximation by step functions, we deduce that \eqref{adj} holds for every $f\in \dom(C_{\A^i})$ that is nonnegative and $f=\varDelta\widetilde f$ with $\widetilde f\in L^2(\MUN{m})$. For every cylindrical $\C$-valued function $f\in\dom(C_{\A^i})$ its module is also a cylindrical function, $|f|\in\dom(C_{\A^i})$ and $\varSigma_{s(p)}(f)=\varSigma_{s(p)}(|f|)$. Hence, by the inequality
\begin{align*}
\big|\is{C_{\A^i} f}{\chi_{\omega}}\big|&\Le \is{C_{\A^i} |f|}{\chi_{\omega}},
\end{align*}
we get the claim.

\noindent{\bf Step 2. }{\em For all $i\in I_{n+r}$, $m\in\N$ and $\omega\in\bor{\R^m}$, if
\begin{align}\label{claim-2.5+}
\limsup_{l\to\infty} \int_{\R^{s(l)}} \chi_{\A_{l}^i(\omega\times \R^{s(l)-m})}\hsf_{\A_{l}^i}^2\D\MUN{s(l)}<+\infty,
\end{align}
then $\chi_{\omega\times\R^\infty}\in \dom(C_{\A^i}^*)$. Moreover, if $k\in\N$, $i\in I_{n+r}$ and $\sigma\in \bor{\R^\infty}$ satisfy $\sigma\subseteq \A^{-i}(\varOmega_k\times\R^\infty)$, then $\chi_{\sigma}\in \dom(C_{\A^i}^*)$.}

Fix $i\in I_{n+r}$. Let $m\in\N$ and $\omega\in\bor{\R^m}$. Then, applying Step 1, inequality \eqref{claim-2.5+} and Proposition \ref{rdzen2}, we deduce that $\chi_{\omega\times\R^\infty}\in \dom(C_{\A^i}^*)$.

Now fix $k\in\N$. It follows from condition (vi) that
\begin{align*}
\munfty\bigg(\Big(\A^{-i}(\varOmega_k\times\R^\infty)\Big)\triangle \Big(\A^{-i}_p(\varOmega_k\times\R^{s(p)-s(\kappa)})\times \R^\infty\Big)\bigg)=0,\quad p\Ge\widetilde p,
\end{align*}
and thus, by condition (v), we have
\begin{multline*}
\limsup_{l\to\infty} \int_{\R^{s(l)}}\chi_{\A^i_l(\A^{-i}_l(\varOmega_k\times\R^{s(l)-s(\kappa)}))}\hsf_{\A_{l}^i}^2\D\MUN{s(l)}\\
=\limsup_{l\to\infty}
\int_{\R^{s(l)}}\chi_{\varOmega_k\times\R^{s(l)-s(\kappa)}}\hsf_{\A_{l}^i}^2\D\MUN{s(l)}
<+\infty.
\end{multline*}
Hence, $\chi_{\A^{-i}(\varOmega_k\times\R^\infty)}\in \dom(C_{\A^i}^*)$ by the first part of the claim. Now, that $\chi_\sigma\in \dom(C_{\A^i}^*)$ follows easily from \eqref{adjoint} and the definition of the domain of a weighted composition operator.

\noindent{\bf Step 3. }{\em For all $i\in I_{n+r}$ and $k\in\N$, and $\{l_j\}_{j=1}^\infty\subseteq \N$ such that $\lim_{j\to\infty}l_j=\infty$, we have
\begin{align}\label{liminf}
\|\chi_{\varOmega_k\times\R^{\infty}} \hsf_{\A^i}\|\Le\limsup_{j\to\infty}\|\chi_{\varOmega_k\times\R^{s(l_j)-s(\kappa)}} \hsf_{\A_{l_j}^i}\|.
\end{align}}

Fix $i\in I_{n+r}$ and $k\in\N$. By Step 2, $\chi_{\A^{-i}(\varOmega_k\times\R^\infty)}\in \dom(C_{\A^i}^*)$, and thus using \eqref{adjoint} we get
\begin{align}\label{step4+}
\is{f}{\chi_{\varOmega_k\times\R^{\infty}} \hsf_{\A^i}}
=\is{f}{C_{\A^i}^* \chi_{\A^{-i}(\varOmega_k\times\R^{\infty})}}
=\is{C_{\A^i} f}{\chi_{\A^{-i}(\varOmega_k\times\R^{\infty})}},\quad f\in\dom(C_{\A^i}).
\end{align}
According to Lemma \ref{densepower}, equalities in \eqref{step4+} are satisfied for every cylindrical step function $f$. Moreover, for every cylindrical step function $f$, by condition (vi) and Step 1, we have
\begin{align*}
\big|\is{C_{\A^i} f}{\chi_{\A^{-i}(\varOmega_k\times\R^{\infty})}}\big|
\Le
\|f\| \limsup_{j\to\infty} \bigg(\int_{\R^{s(l_j)}} \chi_{\A_{l_j}^i(\A_{l_j}^{-i}(\varOmega_k\times \R^{s(l_j)-s(\kappa)}))} \hsf_{\A_{l_j}^i}^2\D\MUN{s(l_j)}\bigg)^{1/2},
\end{align*}
which together with \eqref{step4+} gives
\begin{align*}
\big|\is{f}{\chi_{\varOmega_k\times\R^{\infty}} \hsf_{\A^i}}\big|\Le \|f\|\limsup_{j\to\infty}\|\chi_{\varOmega_k\times\R^{s(l_j)-s(\kappa)}} \hsf_{\A_{l_j}^i}\|,
\end{align*}
for every cylindrical step function $f$. Since every function in $L^2(\munfty)$ may be approximated by cylindrical step functions, we get
\begin{align*}
\|\chi_{\varOmega_k\times\R^{\infty}} \hsf_{\A^i}\|^2\Le \| \chi_{\varOmega_k\times\R^{\infty}} \hsf_{\A^i}\|\limsup_{j\to\infty}\| \chi_{\varOmega_k\times\R^{s(l_j)-s(\kappa)}}\hsf_{\A_{l_j}^i}\|,
\end{align*}
which proves \eqref{liminf}.

\noindent{\bf Step 4. }{\em There exists an injective increasing sequence $\{l_j\}_{j=1}^\infty\subseteq \N$, such that for all $i\in I_{n+r}$ and $k\in\N$,
\begin{align}\label{limL2+}
\lim_{j\to\infty}\| \chi_{\varOmega_k\times\R^{\infty}} \big(\hsf_{\A^i}-\varDelta\hsf_{\A_{l_j}^i}\big)\|=0.
\end{align}}

First we prove \eqref{limL2+} with $i=k=1$. In view of \eqref{liminf} and (v) there exists an injective increasing sequence $\{l_j^{1,1}\}_{j=1}^\infty\subseteq \N$ such that
\begin{align*}
\lim_{j\to\infty}\| \chi_{\varOmega_1\times\R^{\infty}}\varDelta\hsf_{\A_{l_j^{1,1}}}\|=\| \chi_{\varOmega_1\times\R^{\infty}}\hsf_{\A}\|.
\end{align*}
Thus, by \cite[Exercise 4.21(a)]{wei}, it suffices to show that
\begin{align}\label{weakfulli=k=1}
\int_{\varOmega_1\times\R^{\infty}} f\hsf_{\A}\D \munfty=\lim_{j\to\infty} \int_{\varOmega_1\times\R^{\infty}} f\varDelta\hsf_{\A_{l_j^{1,1}}}\D\munfty,\quad  f\in L^2(\munfty).
\end{align}
By condition (vi), equality in \eqref{weakfulli=k=1} holds when $f$ is a cylindrical step function. Since cylindrical step functions are dense in $L^2(\munfty)$ and $\sup_{j\in\N}\|\chi_{\varOmega_1\times\R^{\infty}}\varDelta \hsf_{\A_{l_j^{1,1}}}\|_{L^2(\munfty)}<\infty$, we get \eqref{weakfulli=k=1} and consequently \eqref{limL2+} for $i=k=1$. Repeating the same argument $(n+r-1)$ times (apply Step 3 to consecutive subsequences), we show that there exists a subsequence $\{l_j^{n+r,1}\}_{j=1}^\infty\subseteq\{l_j^{1,1}\}_{j=1}^{\infty}$ such that \eqref{limL2+} is satisfied with $i=1,\ldots , n+r$ and $k=1$. In a similar manner we show that for any $k_0\in\N$ we can find subsequences $\{l_j^{n+r,k_0}\}_{j=1}^\infty\subseteq\ldots\subseteq\{l_j^{n+r,2}\}_{j=1}^\infty\subseteq\{l_j^{n+r,1}\}_{j=1}^\infty$ such that \eqref{limL2+} is satisfied with $i=1,\ldots , n+r$ and $k=1,\ldots,k_0$. Now, the sequence $\{l_j\}_{j=1}^\infty$ given by $l_j=l_j^{n+r,j}$ does the job.

\noindent{\bf Step 5. }{\em $\ca\in\EuScript{S}_{n,r}^*$.}

In view of Lemma \ref{densepower}, \eqref{adjoint} and \eqref{zgoda}, $\ca$ is densely defined in $L^2(\munfty)$ and $\ca^*=W_{\A^{-1}, \hsf_\A}$. In turn, by  \eqref{adjoint}, for every $k\in\N$, $C_{\A_{k}}^*=W_{\A_{k}^{-1}, \hsf_{\A_{k}}}$.
Set
\begin{align*}
S=W_{\A^{-1},\hsf_\A},\quad \text{and}\quad S_{k}=W_{\A_{l_k}^{-1}, \hsf_{\A_{l_k}}},\quad k\in\N,
\end{align*}
where $\{l_k\}_{k\in\N}$ is as in Step 4 with additional requirement that $l_1\Ge \kappa$.

We denote by $\bb$ the family of Borel subsets of $\R^\infty$ defined as follows: $\sigma\in\bb$ if and only if there exists $\widetilde\sigma\in\borc{\R^\infty}$ such that $\munfty(\sigma\triangle\widetilde\sigma)=0$ and either $\widetilde\sigma=\R^\infty$ or there exists $k_1,\ldots,k_{n+r}\in\N$ such that $\widetilde\sigma\subseteq \A^{-1}(\varOmega_{k_1}\times \R^\infty)\cap\ldots\cap \A^{-(n+r)}(\varOmega_{k_{n+r}}\times\R^\infty)$. According to condition (vi) and the fact that $\varOmega_l\nearrow \R^\kappa$ as $l\to\infty$ the family $\bb$ satisfies conditions (i)-(iii) of Proposition \ref{rdzen2}. Indeed, conditions (i) and (iii) are clear. For the proof of (ii) we take $m\in\N$ and $\sigma\in\bor{\R^m}$, and for every $k\in\N$ we consider the sets
\begin{align*}
\omega_k=\big(\sigma\times \R^\infty\big)\cap \Big(\A^{-1}\big(\varOmega_{k}\times\R^\infty\big)\cap \ldots\cap \A^{-(n+r)}\big(\varOmega_{k}\times\R^\infty\big)\Big).
\end{align*}
and
\begin{align*}
\widetilde\omega_{k,p}=\Big( \big(\sigma\times \R^{s(p)-m}\big)\cap \A^{-1}_p\big(\varOmega_{k}\times\R^{s(p)-m}\big)\cap\ldots \cap   \A^{-(n+r)}_p\big(\varOmega_{k}\times\R^{s(p)-m}\big)\Big)\times\R^\infty,\quad p\in\N.
\end{align*}
Then, using (vi), we deduce that $\munfty(\omega_k\triangle \widetilde\omega_{k,p})=0$ for sufficiently large $p\in\N$.
Hence for every $k,p\in\N$, $\omega_k\in\bb$ and $\widetilde\omega_{l,p}\nearrow \sigma\times\R^\infty$ as $l\to\infty$. This and the fact that $\borc{\R^\infty}$ generates $\bor{\R^\infty}$ prove (ii).

Let $\xx$ be a family composed of all characteristic functions $\chi_\sigma\in L^2(\munfty)$ of sets $\sigma\in\bb$. Then by Step 2 we have
\begin{align}\label{x}
\xx\subseteq\bigcap_{i=1}^{n+r}\dom(C_{\A^i}^*).
\end{align}
Thus, by Lemma \ref{adjointpowers}\,(i), $\xx\subseteq \dom(S^{n+r})$.

For $k\in\N$, let $\bb_k$ denote the family of sets $\sigma\in \bor{\R^{s(l_k)}}$ such that either $\sigma=\R^{s(l_k)}$, or $\sigma\subseteq \A_{l_k}^{-1}(\varOmega_{m_1}\times \R^{s(l_k)-s(\kappa)})\cap\ldots\cap \A_{l_k}^{-(n+r)}(\varOmega_{m_{n+r}}\times\R^{s(l_k)-s(\kappa)})$ for some $m_1, \ldots, m_{n+r}\in\N$. Applying condition (v), \eqref{adjoint} and Lemma \ref{adjointpowers}\,(i), we show that $\xx_k\subseteq \dom(S_k^{n+r})$ for every sufficiently large $k\in\N$, where $\xx_k$ is composed of all characteristic functions $\chi_\sigma\in L^2(\MUN{s(l_k)})$ of sets $\sigma\in\bb_k$.

By condition (vi), $\xx=\bigcup_{i=1}^\infty \bigcap_{k=i}^\infty \varDelta\xx_k$. In view of \eqref{x}, \eqref{adjoint} and Proposition \ref{rdzen2}, $\lin\xx$ is a core for $(C_{\A}^*, \ldots, C_{\A^{n+r}}^*)$. Thus, by Lemma \ref{adjointpowers}, $\lin\xx$ is a core for $(S, \ldots, S^{n+r})$. Notice that, by condition (vii), for any $x\in\xx_m$ with $m\in\N$, and $i\in I_{n+r}$, we have $(\varDelta x)\circ \A^{-i}=\lim_{p\to\infty} (\varDelta x)\circ \A_{p}^{-i}$. Hence, by Lemma \ref{adjointpowers} and Step 4, we have
\begin{align*}
S^i \varDelta x &=\hsf_{\A^i}\cdot \Big(\varDelta x\circ \A^{-i}\Big)=\lim_{k\to\infty} \varDelta \hsf_{\A_{l_k}^i}\cdot \Big(\varDelta x\circ \A_{l_k}^{-i}\Big)\\
&=\lim_{k\to\infty} S^i_k \varDelta x,\quad i\in I_{n+r}\text{ and }x\in\xx_m, m\in\N,
\end{align*}
which implies that
\begin{align}\label{kurczak}
\is{S^i \varDelta f}{S^j \varDelta g}=\lim_{k\to\infty}\is{S_k^i \varDelta f}{S_k^j \varDelta g},\quad  \text{for all $\varDelta f, \varDelta g\in \lin\xx$ and $i,j\in I_{n+r}$}.
\end{align}
By subnormality of $S_k$ and Proposition \ref{chsub} we have $S_k\in\EuScript{S}_{n,r}$ for every $k\in\N$. This and \eqref{kurczak} imply that inequality in \eqref{n3} (with $S$ in place of $T$) is satisfied for all $\{f_i^k\colon i=1,\ldots, m,\ k=0,\ldots, r\}\subseteq\xx$. This and the fact that $\lin\xx$ is a core for $(S, \ldots, S^{n+r})$ imply that $\ca\in\EuScript{S}_{n,r}^*$.
\end{proof}
\begin{prop}\label{zale-szefa}
Let $n,r\in\Z_+$. Let $\a=(a_{ij})_{i,j\in\N}\subseteq\R$ be a matrix such that $\munfty\big(\dom_\a\big)=1$, and let $\s=\{s(k)\}_{k=1}^\infty\subseteq \N$ be an increasing sequence. Assume that
\begin{enumerate}
\item[(a)] $\A\colon \dom_\a\to \A(\dom_\a)$, the transformation induced by $\a$, is nonsingular and invertible, and $\munfty\big(\A(\dom_\a)\big)=1$,
\item[(b)] $\A^{-1}$ is $\bor{\R^\infty}$-measurable and nonsingular and induced by a matrix $\a^{-1}\in\mathfrak{F}(\s)$,
\item[(c)] there exist $\kappa\in\N$ and $\{\varOmega_{l}\colon l\in\N\}\subseteq \bor{\R^{s(\kappa)}}$ such that  $\varOmega_{l}\nearrow \R^{s(\kappa)}$ as $l\to\infty$ and $\chi_{\varOmega_{k}\times\R^\infty}\hsf_{\A^i}\in L^2(\munfty)$ for all $k\in \N$ and $i\in I_{n+r}$,
\item[(d)] for every $k\in\N$, the matrix $(\a^{-1})_{k}$ is invertible and normal in $\big(\R^{s(k)},|\cdot|\big)$,
\item[(e)] for all $i\in I_{n+r}$ and $k\in\N$, inequality \eqref{glod} is satisfied, where $\A_{l}$ is the transformation of $\R^{s(k)}$ induced by the inverse of $(\a^{-1})_{l}$, $l\in\N$.
\end{enumerate}
Then $\ca\in\EuScript{S}_{n,r}^*$.
\end{prop}
\begin{proof}
We begin by showing that for every $m\in\N$ there exists $\widetilde p\in\N$ such that for all $i\in I_{n+r}$ and $\sigma\in\bor{\R^m}$ condition \eqref{30stopni-1} is satisfied. Indeed, fix $i\in I_{n+r}$, $m\in\N$ and $\sigma\in\bor{\R^m}$. Suppose that there exists $p_0\in\N$ such that
\begin{align*}
s(p_0)\Ge m\quad \text{and}\quad \rank (\a^{-1})_{p_0,p_0+1}=0.
\end{align*}
Set $\widetilde p=p_0$.  Then normality of $(\a^{-1})_{\widetilde p+1}$ implies that $\rank (\a^{-1})_{\widetilde p+1, \widetilde p}=0$. This and surjectivity of $\A^{-1}$ yield \eqref{30stopni-1}. Now, suppose that for every $p\in\N$ such that $s(p)\Ge m$ we have
\begin{align}\label{pizza-1}
\rank (\a^{-1})_{p, p+1}=s(p)-s(p-1).
\end{align}
It is easily seen that
\begin{align*}
\A^{-1}_{p}\Big(\widetilde \sigma\times \R^{s(p)-m}\Big)= \pi_{s(p)}^{s(p+1)} \circ \A^{-1}_{p+1} \Big(\widetilde \sigma\times \R^{s(p)-m}\times \{0\}\times \ldots\times \{0\}\Big),\quad p\Ge \widetilde p,
\end{align*}
where $\widetilde p$ is the smallest integer such that $s(\widetilde p-1)\Ge m$. This and \eqref{pizza-1} imply that
\begin{align}\label{krzyki}
\munfty \bigg(\A^{-1}\Big(\sigma\times \R^{\infty}\Big)\triangle \A^{-1}_{p}\Big(\widetilde \sigma\times \R^{s(p)-m}\Big)\times \R^\infty\bigg)=0,\quad p\Ge \widetilde p.
\end{align}
Using \eqref{krzyki} repeatedly we obtain \eqref{30stopni-1}.

Now, if $m\in\N$, $\sigma\in\R^m$, $i\in I_{n+r}$, then there exists $\widehat p\in \N$ such that for $\munfty$-a.e. $x\in\R^\infty$ we have
\begin{align*}
\Big(\chi_{\sigma\times \R^\infty}\circ \A^{-i}\Big) (x)=1&\Longleftrightarrow\Big(\big(\A^{-i}(x)\big)_1,\ldots,\big(\A^{-i}(x)\big)_m\Big)\in\sigma\\
&\Longleftrightarrow \Big(\big((\A_p^{-i}\circ \pi_p) (x)\big)_1,\ldots,\big((\A_p^{-i}\circ \pi_p) (x)\big)_m\Big)\in\sigma,\quad p\Ge \widehat p\\
&\Longleftrightarrow \varDelta\Big(\chi_{\sigma\times \R^{s(p)-m}}\circ \A_p^{-i}\Big)(x)=1,\quad p\Ge \widehat p
\end{align*}
This implies that $\chi_{\sigma\times \R^\infty}\circ \A^{-i}=\lim_{p\to\infty} \varDelta\Big(\chi_{\sigma\times \R^{s(p)-m}}\circ \A_p^{-i}\Big)$, which all proves that \eqref{3studentki} is satisfied.

Finally, since, by \eqref{choinka} and Theorem \ref{szarlotka}, $C_{\A_k}$ is densely defined and cosubnormal for every $k\in\N$, we can apply Theorem \ref{cosubgauss-}.
\end{proof}
\begin{rem}
Concerning Theorem \ref{cosubgauss-} it is worth noticing that, by Step 3, condition (iii) of Theorem \ref{cosubgauss-} is automatically satisfied if $\limsup_{l\to \infty}\|\hsf_{\A_{l}^i}\|<+\infty$, $i\in I_{n+r}$. 
\end{rem}
Below we provide examples of unbounded composition operators induced by infinite matrices that belong to $\EuScript{S}_{n,r}^*$. The first one, with diagonal matrix symbols, is motivated by \cite[Theorem 4.1]{sto}, where cosubnormality of bounded $\ca$'s of this kind was shown (by use of different methods).
\begin{exa}\label{diag-2}
Let $n,r\in\Z_+$. Let $\a=(a_{ij})_{i,j\in\N}\subseteq \R$ be a diagonal matrix, i.e. $a_{ii}=\alpha_i$ with $\{\alpha_m\}_{m=1}^\infty\subseteq(0,\infty)$, and $a_{ij}=0$ for all $i,j\in\N$ such that $i\neq j$. Assume that
\begin{align*}
0<\alpha_{k}<1\quad \text{for all $k\in\N$ such that $k> n+r$,}
\end{align*}
and
\begin{align}\label{absconv}
\text{$\sum_{k=1}^\infty \big(1-\alpha_k\big)$ is convergent.}
\end{align}
Clearly, $\a$ induces $\aa$-measurable invertible transformation $\A$ of $\R^\infty$ such that $\dom_\a$ is a set of full measure $\munfty$. Let $\s=\{s(k)\}_{k=1}^\infty\subseteq \N$ be given by $s(k)=k$. Then $\A^{-1}$ is $\aa$-measurable transformation induced by a matrix $\a^{-1}\in\frs(\s)$ such that $\dom_{\a^{-1}}$ is a set of full measure $\munfty$.  Let $\{\varOmega_k\colon k\in\N\}$ be given by $\varOmega_k=[-k,k]^{n+r}$. By \eqref{pochodna}, for all $i\in I_{n+r}$, we have\allowdisplaybreaks
\begin{align*}
\hsf_{\A_{l}^i}(x_1,\ldots,x_l)
&=\frac{1}{\alpha_1^i\cdots\alpha_l^i}\exp\frac{|(x_1,\ldots,x_l)|^2-\big|\big(\alpha_1^{-i}x_1,\ldots,\alpha_l^{-i} x_l\big)\big|^2}{2}\\
&=\prod_{j=1}^l\frac{1}{\alpha_j^i}\exp\Big(-\frac{x_j^2}{2}\frac{1-\alpha_j^{2i}}{\alpha_j^{2i}}\Big),\quad (x_1,\ldots,x_l)\in\R^l,\ l\in\N,
\end{align*}
where $\A_l$ is the transformation of $\R^l$ induced by the matrix $\a_l$, $l\in\N$. This together with the change-of-variable theorem (cf. \cite[Theorem 8.26]{rud}) and \eqref{absconv} implies that there is $C>0$ such that\allowdisplaybreaks
\begin{align*}
\|\chi_{\varOmega_k\times\R^{l-n-r}}\hsf_{A_{l}^i}\|_{L^2(\MUN{l})}^2 &
=\prod_{j=1}^{n+r} \frac{1}{\alpha_j^{2i}}\int_{-k}^k\mathrm{e}^{-\frac{x_j^2(1-\alpha_j^{2i})}{\alpha_j^{2i}}}\D\MUN{1}
\prod_{j=n+r+1}^l \frac{1}{\alpha_j^{2i}}\int_{-\infty}^\infty\mathrm{e}^{-\frac{x_j^2(1-\alpha_j^{2i})}{\alpha_j^{2i}}}\D\MUN{1}\\
&=C \prod_{j=n+r+1}^l \frac{1}{\alpha_j^{2i}} \frac{1}{\sqrt{2\pi}}\int_{-\infty}^\infty \mathrm{e}^{-\frac{x_j^2(2-\alpha_j^{2i})}{2\alpha_j^{2i}}}\D \M_1\\
&= C \prod_{j=n+r+1}^l \Big(\alpha_j^i\sqrt{2-\alpha_j^{2i}}\Big)^{-1},\quad k,l\in\N \text{ and }l>n+r.
\end{align*}
Since $\alpha_j^i\sqrt{2-\alpha_j^{2i}}<1$ for every $j>n+r$ and every $i\in I_{n+r}$, we infer from \eqref{absconv} and \cite[Chapter VII]{kno}) that $0<\lim_{l\to\infty}\|\chi_{\varOmega_k\times \R^{l-n-r}}\hsf_{A_{l}^i}\|_{L^2(\MUN{l})}<\infty$. In view of \cite[Corollary 5.1]{bud-pla-s2}, the above discussion shows that conditions (a), (b) and (d) of Proposition \ref{zale-szefa} are satisfied. Now, set
\begin{align*}
\hsf_i \big(\{x_j\}_{j\in\N}\big)=\prod_{j=1}^\infty\frac{1}{\alpha_j^i}\exp\Big(-\frac{x_j^2}{2}\frac{1-\alpha_j^{2i}}{\alpha_j^{2i}}\Big),\quad \{x_j\}_{j=1}^\infty\in \ell^2(\p_{\alpha,i}), i\in I_{n+r},
\end{align*}
where $\p_{\alpha,i}=\{p_j^{(i)}\}_{j=1}^\infty$ with $p_j^{(i)}=|1-\alpha_j^{2i}|$. Since for every $i\in I_{n+r}$, the product $\prod_{j=1}^\infty \alpha_j^{-i}$ is convergent by \eqref{absconv} and \cite[Theorem 3, p.\ 219]{kno}), and the series $\sum_{j=1}^\infty \frac{x_j^2}{2}\frac{1-\alpha_j^{2i}}{\alpha_j^{2i}}$ is convergent for every $\{x_j\}_{j=1}^\infty\in \ell^2(\p_{\alpha,i})$, we see that $\hsf_i$ is well-defined. By \cite[Lemma 11 and Theorem 1.3, page 92]{be-ko}, $\ell^2(\p_{\alpha,i})$ is a set of full measure $\munfty$. Hence $\hsf_i$ may be treated as a $\bor{\R^\infty}$-measurable mapping defined on the whole of $\R^\infty$. For all $i\in I_{n+r}$ and $k,m\in\N$ and $\sigma\in\bor{\R^m}$, we have
\begin{align*}
&\munfty\circ \A^{-i}\big((\sigma\times\R^\infty)\cap(\varOmega_k\times \R^\infty)\big)
=\lim_{l\to\infty} \MUN{l}\big(\A_l^{-i}\big((\sigma\times \R^{l-m})\cap(\varOmega_k\times \R^{l-n-r})\big)\big)\\
&=\lim_{l\to\infty} \int_{\sigma\times \R^{l-m}}\chi_{\varOmega_k\times \R^{l-n-r}}\hsf_{\A_l^i}\D\MUN{l}=\int_{\sigma\times\R^\infty}\chi_{\varOmega_k\times \R^\infty} \hsf_i \D\munfty,
\end{align*}
where the last equality follows from Lebesgue dominated convergence theorem and the fact that $\chi_{\varOmega_k\times\R^\infty}\hsf_{i}$ is essentially bounded. This, together with \cite[Theorem 10.3]{bil}, implies that for every $i\in I_{n+r}$, $\A^i$ is nonsingular and $\hsf_i=\hsf_{\A^i}$ a.e. $[\munfty]$. Hence the conditions (a) to (e) of Proposition \ref{zale-szefa} are satisfied, and consequently $\ca\in \EuScript{S}_{n,r}^*$.
\end{exa}
\begin{rem}
In view of \cite[Theorem 4.1]{sto}, the operator $\ca$ from Example \ref{diag-2} is bounded and cosubnormal, whenever\footnote{If $\alpha_m>1$ for some $m\in\N$, then $\ca$ is unbounded in $L^2(\munfty)$ (cf. \cite[Proposition 2.2]{sto}).} $\alpha_i\in(0,1)$ for every $i\in\N$. It is worth noticing that, if this is the case, cosubnormality of $\ca$ can also be deduced from Proposition \ref{zale-szefa}. To this end, we first observe that $\ca\in \EuScript{S}_{n,0}^*$ for every $n\in\N$ which follows from Theorem \ref{cosubgauss-}. On the other hand, by applying \cite[Corollary 5.5]{bud-pla-s2}, we see that $\ca$ is a bounded operator on $L^2(\munfty)$. Hence, using Proposition \ref{chsub}, we get cosubnormality of $\ca$.
\end{rem}
\begin{rem}
It should be noted that any nontrivial scalar multiple of the identity mapping on $\R^\infty$ cannot satisfy assumptions of Proposition \ref{zale-szefa}. Indeed, suppose $\A\colon\R^\infty\to\R^\infty$ is given by $\A(x)=\alpha x$, with $\alpha \in\R\setminus\{0,1,-1\}$. Then either $\A$ or $\A^{-1}$ is {\em singular}, i.e., one of them is not nonsigular. To see this we first fix  $\alpha\in(0,1)$. Then we take any sequence $\{x_n\}_{n=1}^\infty\subseteq (0,1)$ such that
\begin{align}\label{pizzajedzie}
\sum_{n=1}^\infty x_n<\infty\quad\text{and}\quad\sum_{n=1}^\infty x_n^{\sqrt{\alpha}}=\infty
\end{align}
Set $a_n=\sqrt{2\ln x_n^{-1}}$, $n\in\N$. Using the well-known method of proving that the gaussian measure $\MUN{1}$ is a probability measure due to Poisson (cf. \cite[p.\ 18-19]{stu}) we can show that
\begin{align}\label{tomek}
1-\exp\Big(-\frac{a^2}{2}\Big)\Le\bigg(\frac{1}{\sqrt{2\pi}}\int_{[-a,a]}\exp\Big(-\frac{x^2}{2}\Big)\D\M_1\bigg)^2\Le 1-\exp(-a^2),\quad a\in(0,\infty).
\end{align}
By \eqref{pizzajedzie} and \cite[Theorem 3, p.\ 219]{kno}) the product $\prod_{n=1}^\infty \Big(1-\exp\big(-\frac{a_n^2}{2}\big)\Big)$ is convergent, which in view of \eqref{tomek} implies that
\begin{align*}
0<\prod_{n=1}^\infty \bigg(1-\exp\Big(-\frac{a_n^2}{2}\Big)\bigg)\Le \bigg(\prod_{n=1}^\infty \frac{1}{\sqrt{2\pi}}\int_{[-a_n,a_n]}\exp\Big(-\frac{x^2}{2}\Big)\D\M_1\bigg)^2\Le 1.
\end{align*}
This, in turn, yields
\begin{align*}
0<\munfty\big(\prod_{n=1}^\infty [-a_n,a_n]\big).
\end{align*}
On the other hand, the product $\prod_{n=1}^\infty \Big(1-\exp\big(-\alpha^2 a_n^2\big)\Big)
$ is divergent to $0$ (use again \eqref{pizzajedzie} and \cite[Theorem 3, p.\ 219]{kno}). Hence, by \eqref{tomek}, we deduce that
\begin{align*}
0=\munfty\big(\prod_{n=1}^\infty [-\alpha a_n,\alpha a_n]\big),
\end{align*}
which shows that $\A$ is singular. Similar reasoning proves the same for other $\alpha$'s belonging to $\R\setminus\{0,1-1\}$ (if $|\alpha|>1$, then $\A^{-1}$ is singular).

In fact, modifying slightly the above argument, one can show that any transformation $\A$ of $\R^\infty$ given by $\A\big(\{x_k\}_{k=1}^\infty\big)=\{\alpha_k x_k\}_{k=1}^\infty$, with $\{\alpha_k\}_{k=1}^\infty\subseteq\R$ such that either $\limsup_{k\to\infty} |\alpha_k|\neq 1$ or $\liminf_{k\to\infty}|\alpha_k|\neq 1$, doesn't satisfy assumption of Proposition \ref{zale-szefa}.
\end{rem}
Below, in Proposition \ref{cosubgauss-positive}, we provide another set of conditions implying that $\a$ satisfies the assumptions of Theorem \ref{cosubgauss-}. To this end we need an auxiliary lemma in which a class of infinite matrices (suitable for Proposition \ref{cosubgauss-positive}) is distinguished. Recall that a matrix $(a_{i,j})_{i,j\in\N}$ is called {\em $\eta$-banded}, with $\eta\in\Z_+$, if $a_{i,j}=0$ for all $i,j\in\N$ such that $|i-j|>\eta$. Below, and later on, $\delta_{ij}$ stands for the Kronecker's delta.
\begin{lem}\label{trace}
Let $\p=\{p_j\}_{j=1}^\infty$ and $\{\alpha_j\}_{j=1}^\infty$ be sequences of positive real numbers such that
\begin{enumerate}
\item[(a)] $\{p_j\}_{j=1}^\infty$ and $\{\alpha_j\}_{j=1}^\infty$ belong to $\ell^1(\N)$,
\item[(b)] there are $0<m<M<\infty$ such that $\frac{p_{j+1}}{p_j}\in (m, M)$ for every $j\in\N$,
\item[(c)] $\sup_{j\in\N}\frac{\alpha_j}{p_j}<\infty$.
\end{enumerate}
Let $\eta\in\Z_+$. Let $\widehat\b=(\widehat b_{i,j})_{i,j\in\N}$ be a $\eta$-banded matrix with real entries such that
\begin{enumerate}
\item[(d)] $|\widehat b_{i,j}|\Le \alpha_i$ for all $i,j\in\N$,
\item[(e)] $\widehat\b$ is symmetric, i.e. $\widehat b_{i,j}= \widehat b_{j,i}$ for all $i,j\in\N$.
\end{enumerate}
Let $\b=(\delta_{ij}+\widehat b_{i,j})_{i,j\in\N}$. Then the following conditions are satisfied
\begin{enumerate}
\item[(1)] $\widehat\b$ induces a trace-class operator $\widehat B$ on $\ell^2(\N)$,
\item[(2)] $\det \b$ is well-defined; moreover, $\det\b\neq 0$ if and only if the operator $I+\widehat B$ is invertible,
\item[(3)] there exists $\widetilde C>0$ such that
\begin{align*}
\Big|\sum_{j=1}^\infty x_j^2-\big(\b x\big)_j^2\Big|\Le \widetilde C \sum_{j=1}^\infty x_j^2 p_j,\quad x=\{x_j\}_{j=1}^\infty\in\R^\infty.
\end{align*}
\item[(4)] if $\det \b\neq 0$, then $\b$ induces an invertible transformation of $\R^\infty$.
\end{enumerate}
\end{lem}
\begin{proof}
If $\{e_i\}_{i=1}^\infty$ is the standard orthonormal basis for $\ell^2(\N)$, then by (d) we have
\begin{align*}
\sum_{i=1}^\infty \big|\is{|\widehat B|e_i}{e_i}\big|\Le \sum_{i=1}^\infty \big\||\widehat B|e_i\big\|= \sum_{i=1}^\infty \big\|U|\widehat B|e_i\big\|\\
= \sum_{i=1}^\infty \big\|\widehat B e_i\big\|\Le (2\eta+1)\sum_{i=1}^\infty \alpha_i<\infty,
\end{align*}
where $\widehat B=U|\widehat B|$ is the polar decomposition of $\widehat B$. This proves that $\widehat B$ is a trace-class operator. This, according to \cite[p. 46]{bot-gru},  imply (2). For the proof of (3) we first observe that if $S$ is the shift operator $S (x_1, x_2, \ldots)=(0, x_1, x_2, \ldots)$ acting on $\ell^2(\p)$, then, by (b), there exists $C>0$ such that $\max_{i\in I_\eta}\{ \|S^i\|, \|S^{*i}\|\}<C$. Let $c=\max\big\{\sup_{j\in\N}\frac{\alpha_j}{p_j}, \sup_{j\in\N}\alpha_j\big\}$. Then for every $\{x_j\}_{j=1}^\infty\in\ell^2(\p)$ we have
\begin{align*}
\Big|\sum_{j=1}^\infty x_j^2-\big(\b x\big)_j^2\Big|
&\Le \sum_{j=1}^\infty \bigg( \big(\widehat\b x\big)_j^2+ 2|x_j|\big|\big(\widehat\b x\big)_j\big| \bigg)
\Le  2\sum_{j=1}^\infty \sum_{k=-\eta}^\eta\alpha_j^2 x_{j+k}^2
+ 2\|x\|_{\ell^2(\p)}\bigg(\sum_{j=1}^\infty \frac{(\widehat\b x)_j^2}{p_j}\bigg)^{1/2}\\
&\Le 2\sum_{k=-\eta}^\eta \sum_{j=1}^\infty \alpha_j^2 x_{j+k}^2
+2\|x\|_{\ell^2(\p)}\bigg(2\sum_{k=-\eta}^\eta \sum_{j=1}^\infty  \frac{\alpha_j^2x_{j+k}^2}{p_j}\bigg)^{1/2}\\
&\Le 2c^2 \sum_{k=-\eta}^\eta\sum_{j=1}^\infty x_{j+k}^2p_j+2\|x\|_{\ell^2(\p)}\bigg(2c^2\sum_{k=-\eta}^\eta \sum_{j=1}^\infty  x_{j+k}^2p_j\bigg)^{1/2}\\
&\Le 2c^2 (2\eta+1) C^2\|x\|_{\ell^2(\p)}^2+2\|x\|_{\ell^2(\p)}\bigg(2c^2(2\eta+1) C^2 \|x\|_{\ell^2(\p)}^2 \bigg)^{1/2}.
\end{align*}
This proves (3). If $\det\b\neq 0$, then, by (2), the matrix $\b$ induces an invertible operator $B$ in $\ell^2(\N)$. Let $\a$ denote the matrix of the operator $B^{-1}$ with respect to the standard orthonormal basis of $\ell^2(\N)$. In view of \cite[Proposition 2.3]{d-m-s-moc-84} and Lemma \ref{exponential}, $\dom_\a$ is a set of full measure $\munfty$. Let $\A$ be the transformation of $\R^\infty$ induced by $\a$, $\B$ be the transformation induced by $\b$. It is a matter of simple calculations (based on the fact that $BB^{-1}=B^{-1}B=I$) to show that $\A\B$ and $\B\A$ coincide with the identity mapping on $\R^\infty$ up to sets of full measure $\munfty$.
\end{proof}
Regarding Lemma \ref{trace}, we note that if $\widehat\b$ is a matrix satisfying assumptions (a) and (d) of Lemma \ref{trace}, then for every $k \in \N$ we have
\begin{align*}
\Big|\Big(\widehat b^k\Big)_{i,j}\Big| \Le   \Big(\sum_{l=1}^{\infty} \alpha_l\Big)^{k-1} \alpha_i, \quad i,j \in \N.
\end{align*}
As a consequence, we get the following.
\begin{cor} \label{slub}
Under the assumptions of Lemma \ref{trace} for every $k \in \N$ we have $(\widehat \b)^k$ is a trace-class operator, $\det \b^k$ is well-defined and there exists $\widetilde{C}>0$ such that
\begin{align}\label{ogorkowa}
\Big|\sum_{j=1}^\infty x_j^2-\big(\b^k x\big)_j^2\Big|\Le \widetilde C \|x\|_{\ell^2(\p)}^2,\quad x=\{x_j\}_{j=1}^\infty\in\ell^2(\p).
\end{align}
\end{cor}
\begin{prop}\label{cosubgauss-positive}
Let $n,r\in\Z_+$ and $\eta\in\N$. Let $\widehat\b=(\widehat b_{ij})_{i,j\in\N}$, $\p=\{p_j\}_{j=1}^\infty$ and $\{\alpha_j\}_{j=1}^\infty$ satisfy conditions {\em (a)} to {\em(e)} of Lemma \ref{trace}. Denote by $\b$ the matrix $(\delta_{ij}+\widehat b_{ij})_{i,j\in\N}$ and by $\B$ the transformation of $\R^\infty$ induced by $\b$. Let $\s=\{s(k)\}_{k=1}^\infty \subseteq \N$ be an increasing sequence. Assume that
\begin{enumerate}
\item[(f)] there exists $\rho>0$ such that $|\det \b_{k}| \in [\rho,+\infty)$ for all $k\in\N$.
\end{enumerate}
Then $\B$ is invertible, $\B^{-1}$ is nonsingular and $C_{\B^{-1}}\in \EuScript{S}_{n,r}^*$.
\end{prop}
\begin{proof}
Note that, by \cite[p. 46]{bot-gru}, (f) and Lemma \ref{trace}, $|\det\b|=\lim_{k\to\infty}|\det\b_k|\in[\rho,\infty)$ and $\det (\b^i)_l\neq 0$ for $i\in I_{n+r}$ and sufficiently large $l\in\N$. Hence, by (4) of Lemma \ref{trace}, the transformation $\B$ is invertible. According to (3) of Lemma \ref{trace} and Corollary \ref{slub} there exists $\widetilde{C}>0$ such that the inequality \eqref{ogorkowa} holds for every $k \in I_{n+r}$. Hence the formula
\begin{align*}
\hsf_i \Big(\{x_j\}_{j=1}^\infty\Big)=|\det \b|^{i} \exp \frac{1}{2}\Big(\sum_{j=1}^\infty x_j^2 -(\b^i x)_j^2 \Big),\quad \{x_j\}_{j=1}^\infty\in \ell^2(\p),
\end{align*}
defines a $\bor{\R^\infty}$-measurable function on $\R^\infty$ for every $i\in I_{n+r}$. Now, let us choose $\kappa\in \N$ so that $1-2\widetilde C p_j >0$ for every $j> s(\kappa)$, and let $\{\varOmega_k\colon k\in\N\}$ be given by $\varOmega_k=[-k,k]^{s(\kappa)}$.
Then, employing change-of-variable theorem and the fact that $\ell^2(\p)$ is a set of full measure $\munfty$ in $\R^\infty$, we get
\begin{align}\label{exarnd1+}
&\int_{\varOmega_k\times\R^\infty} \exp \Big(\sum_{j=1}^\infty x_j^2 -(\b^i x)_j^2 \Big)\D\munfty \notag
\Le \int_{\varOmega_k\times\R^\infty} \exp \Big(\widetilde C \big\|\{x_j\}_{j=1}^\infty\big\|^2_{\ell^2(\p)} \Big)\D\munfty\\\notag
&= \int_{\varOmega_k} \exp \Big(\widetilde C \sum_{j=1}^{s(\kappa)} x_j^2p_j \Big)\D\MUN{s(\kappa)} \cdot \prod_{j=s(\kappa)+1}^\infty \frac{1}{\sqrt{2\pi}}\int_\R \exp\Big(-\frac{(1-2\widetilde Cp_j)x_j^2}2\Big)\D \M_1 \\
&=\int_{\varOmega_k} \exp \Big(\widetilde C \sum_{j=1}^{s(\kappa)} x_j^2p_j \Big)\D\MUN{s(\kappa)} \cdot \prod_{j=s(\kappa)+1}^\infty \frac{1}{\sqrt{1-2\widetilde Cp_j}},\quad i\in I_{n+r}, k\in\N.
\end{align}
Since $\p\in\ell^1(\N)$, we deduce that $\prod_{j=s(\kappa)+1}^\infty \frac{1}{\sqrt{1-2\widetilde Cp_j}}$ is convergent (cf. \cite[Chapter VII]{kno}). Thus, by \eqref{exarnd1+}, the function $\chi_{\varOmega_k\times\R^\infty}\hsf_{i}$ belongs to $L^2(\munfty)$ for all $i\in I_{n+r}$ and $k\in\N$. Note that there exists $C>0$ such that
\begin{align}\label{ogorkowa_z_maka}
\Big|\sum_{j=1}^{s(l)} x_j^2-\big((\b^k)_l x\big)_j^2\Big|\Le C \sum_{j=1}^{s(l)} x_j^2 p_j,\quad x\in \R^{s(l)}, l \in \N, k \in I_{n+r},
\end{align}
(see the proof of assertion (3) of Lemma \ref{trace}). Hence arguing as in \eqref{exarnd1+} we show that the function $\chi_{\varOmega_k\times\R^\infty}\varDelta \hsf_{\B_l^{-i}}$ belongs to $L^2(\MUN{s(l)})$ for all $i\in I_{n+r}$, $k\in\N$ and every $l\in \N$ such that $l\Ge \kappa$, where $\B_l^i$ is a transformation of $\R^{s(l)}$ induced by $(\b^i)_l$. Since, for every $x=\{x_i\}_{i=1}^\infty\in \ell^2(\p)$ we have
\begin{multline*}
\Big| \sum_{j=1}^{s(l)}\Big( x_j^2 - \big((\b^i)_l \pi_{s(l)}x\big)^2_j \Big)  -  \sum_{j=1}^{\infty}\Big( x_j^2 - (\b^i x)^2_j \Big)\Big| \\
= \Big| \sum_{j = s(l)- i\eta}^{s(l)} x_j^2 - \big((\b^i)_l \pi_{s(l)}x \big)^2_j  \Big| + \Big| \sum_{j = s(l)- i\eta}^{\infty} x_j^2 - (\b^i x)^2_j  \Big|.
\end{multline*}
This combined with \eqref{ogorkowa} and \eqref{ogorkowa_z_maka} proves that
\begin{align*}
\hsf_i (x) =\lim_{l\to\infty} \varDelta \hsf_{\B_l^{-i}}(x), \quad x \in \ell^2(\p).
\end{align*}
By \eqref{exarnd1+} and \eqref{ogorkowa_z_maka}, the function $H_{k,i}\colon \ell^2(\p)\to\C$ given by
\begin{align*}
H_{k,i}(x)=\sup_{l\in\N}\Big\{\big|\det (\b^i)_l\big|\Big\} \chi_{\varOmega_k\times \R^\infty}\exp\tfrac12\Big( C\|x\|^2_{\ell^2(\p)}\Big)
\end{align*}
is a majorant in $L^2(\munfty)$ for the sequence $\{\chi_{\varOmega_k\times \R^\infty } \varDelta \hsf_{\B_l^{-i}}\}_{l=1}^\infty$, $k\in\N$ and $i\in I_{n+r}$. Hence, by the Lebesgue dominated convergence theorem we get
\begin{align}\label{debil}
\chi_{\varOmega_k\times\R^\infty}\hsf_i=\lim_{l\to\infty} \chi_{\varOmega_k\times\R^\infty} \varDelta \hsf_{\B_l^{-i}},\quad k\in\N, i\in I_{n+r}.
\end{align}
Since for all $i\in I_{n+r}$ and $l\in\N$, the matrices $\b^i$ are banded and the matrices $(\b^i)_l$ are invertible, we deduce from \eqref{debil} that for all $i\in I_{n+r}$, $k,m\in\N$, and $\sigma\in\bor{\R^m}$, we have
\begin{align*}
\munfty\bigg( \B^i \big((\sigma\times\R^\infty)\cap(\varOmega_k\times \R^\infty)\big)\bigg)
&=\lim_{l\to\infty} \MUN{s(l)}\bigg(\B_l^i\big(\big(\sigma\times \R^{s(l)-m})\cap(\varOmega_k\times \R^{s(l)-s(\kappa)}\big)\big)\bigg)\\
&=\lim_{l\to\infty} \int_{\sigma\times \R^{s(l)-m}}\chi_{\varOmega_k\times \R^{s(l)-s(\kappa)}}\hsf_{\B_l^{-i}}\D\MUN{s(l)}\\
&=\int_{\sigma\times\R^\infty}\chi_{\varOmega_k\times \R^\infty} \hsf_i \D\munfty.
\end{align*}
This, by \cite[Theorem 10.3]{bil}, implies that for every $i\in I_{n+r}$, the transformation $\B^{-i}$ is nonsingular and $\hsf_i=\hsf_{\B^{-i}}$ a.e. $[\munfty]$. This, Theorem \ref{szarlotka}, Lemma \ref{exponential} and equality \eqref{debil}  imply that conditions (i) to (vii) of Theorem \ref{cosubgauss-} hold with $\A=\B^{-1}$ and $\A_l=\B_l^{-1}$, $l\in\N$ (conditions (vi) and (vii) follow easily from $\eta$-boundedness of $\b$). Hence, applying Theorem \ref{cosubgauss-}, we get that $C_{\B^{-1}}$ belongs to $\EuScript{S}_{n,r}^*$.
\end{proof}
\begin{rem}
Regarding Proposition \ref{cosubgauss-positive}, we note that there is an another way of producing the inverses of an $\eta$-bounded matrix $\b$ and a transformation $\B$ via inductive technique. To this end, instead of conditions (a)-(e) of Lemma \ref{trace}, we assume that
\begin{enumerate}
\item[(1)] there exists $\rho>0$ such that $\sigma(\b_{k}) \subseteq [\rho,+\infty)$ for all $k\in\N$,
\item[(2)] for every $k\in\N$, $\b_{k}$ is normal in $(\R^{s(k)},|\cdot|)$,
\item[(3)] for every $\varepsilon >0$ there is $k_0\in\N$ such that for every $m\Ge l\Ge k_0$
\begin{align*}
|\b_{m}\iota_{s(l)}^{s(m)} x-\iota_{s(l)}^{s(m)}\b_{l} x|\Le \varepsilon \big(|x|+|\b_{m}\iota_{s(l)}^{s(m)} x|+|\b_{l}x|\big),\quad x\in\R^{s(l)},
\end{align*}
where $\iota^s_t\colon \R^t\ni(x_1,\ldots, x_t)\mapsto (x_1,\ldots, x_t, 0,\ldots, 0)\in\R^s$ for $t\Le s$.
\end{enumerate}
Then we proceed as follows. For $j\in\N$, let $\B_j$ denote the linear mapping $\R^{s(j)}\to\R^{s(j)}$ induced by $\b_j$. By \cite[Theorem 1.1, Lemma 4.3]{jan} and (3), the sequence $\{\B_{j}\}_{j=1}^\infty$ induces a closable densely defined operator $B_\infty$ acting on $\ell^2(\N)$ according to the formula
\begin{gather*}
\dom(B_\infty)=\bigcup_{k\in\N}\{\iota_k x\colon x\in\R^k \text{ such that } \lim_{j\to\infty} \iota_{s(j)} \b_{j}\iota_k^{s(j)} x \in\ell^2(\N)\} \\
B_\infty \iota_k x=\lim_{j\to\infty} \iota_{s(j)} \b_{j} \iota_k^{s(j)} x,\quad \iota_k x\in \dom(B_\infty),
\end{gather*}
where $\iota_t\colon \R^t\ni(x_1,\ldots, x_t)\mapsto (x_1,\ldots, x_t, 0,\ldots)\in\R^\infty$ for $t\in\N$. $B_\infty$ is the inductive limit of $\{\B_{j}\}_{j=1}^\infty$. In view of \cite[Corollary 2.2, Lemma 4.3]{jan} and (1), $\sigma(B_\infty)\subseteq [\rho,+\infty)$ and thus $B_\infty$ is invertible. From this point we proceed as in the proof of (4) of Lemma \ref{trace}.
\end{rem}
With help of Proposition \ref{cosubgauss-positive} we can deliver an example of composition operators belonging to $\EuScript{S}_{n,r}^*$ induced by non-diagonal transformation of $\R^\infty$. The operator is induced by a 1-banded matrix.
\begin{exa}
Let $n,r\in\Z_+$. Let $\b$ be a matrix given by
\begin{align*}\b=
\left(
  \begin{array}{cccccc}
    1           &q^1    &0      &0      &0          &\cdots \\
    q^1         &1      &q^2    &0      &0          &\cdots \\
    0           &q^2    &1      &q^3    &0          &\cdots \\
    0           &0      &q^3    &1      &q^4        &\cdots \\
    \cdots      &\cdots &\cdots &\cdots &\cdots     &\cdots\\
  \end{array}
\right),
\end{align*}
with $q\in \big(0,\frac{\sqrt2}{2}\big)$. Let $\{s(k)\}_{k=1}^\infty$ be given by $s(k)=k$. Set $\p=\{q^j\}_{j=1}^\infty$ and define $\alpha_j=q^{j-1}$, $j\in\N$. Clearly, $\det \b_l = \det \b_{l-1} - q^{2l-2} \det \b_{l-2}$, $l\Ge 3$, which implies that $ 1 - \sum_{k=2}^{\infty} q^{2k-2} < \det \b_{l} <1$ for $l \Ge 2$. As a consequence, the assumptions of Proposition  \ref{cosubgauss-positive} are satisfied. Hence $\b$ induces an invertible transformation $\B$ of $\R^\infty$ such that $\B^{-1}$ is nonsingular and $C_{\B^{-1}}$ belongs to $\EuScript{S}^*_{n,r}$. If fact, $C_{\B^{-1}}\in\EuScript{S}^*_{n,r}$ for every $n,r\in \Z_+$.

Now, let $\xx\subseteq L^2(\munfty)$ be the family defined as in the Step 5 of proof of Theorem \ref{cosubgauss-} with $\A=\B^{-1}$. According to Lemma \ref{adjointpowers} and \eqref{x} we see that $\xx \subseteq \dom\Big( \big(C^*_{B^{-1}}\big)^{n} \Big)$ for every $n \in \Z_+$. Arguing as in the proof of Step 5 of Theorem \ref{cosubgauss-} and using \cite[Lemma 3.2]{b-d-p-matrix-subn} we see that $\xx$ is linearly dense in $L^2(\munfty)$. Thus every power of $C^*_{B^{-1}}$ is densely defined. This combined with \cite[Theorem 52]{b-j-j-s-wco} proves that $\dom^\infty(C^*_{B^{-1}})$ is dense in $L^2(\munfty)$. On the other hand, by Proposition \ref{chsub}, the operator $C^*_{B^{-1}}|_{\dom^\infty(C^*_{B^{-1}})}$ is subnormal.
\end{exa}
\begin{rem}
Concluding the paper we point out that Theorem \ref{cosubgauss-} and Propositions \ref{zale-szefa} and \ref{cosubgauss-positive} rely essentially on the very precise knowledge of the Radon-Nikodym derivative $\hsf_\A$, which is due to the inequality \eqref{cosubgauss-positive}. It seems desirable to look for some inductive-limit-based criteria for cosubnormality, which would be independent of the knowledge of $\hsf_\A$.
\end{rem}

\end{document}